\documentclass[reqno,11pt]{amsart}
\usepackage{amsmath,amssymb,mathrsfs,amsthm,amsfonts}
\usepackage[inline]{enumitem} 
\usepackage[usenames,dvipsnames]{xcolor}
\usepackage{hyperref}
\usepackage{comment}
\usepackage{array,longtable,booktabs,caption}
\DeclareMathAlphabet{\mathpzc}{OT1}{pzc}{m}{it}
\usepackage{stmaryrd}

\hypersetup{%
	colorlinks=true, linkcolor=blue,
	citecolor=ForestGreen
}
\usepackage[paper=letterpaper,margin=1.22in, top=1in]{geometry}
\usepackage{graphicx}
\usepackage{mathrsfs}
\usepackage{listings}
\newcommand{\ind}{\mathbf{1}}
\usepackage{bbm}
\usepackage{tikz}
\usetikzlibrary{arrows}

\newtheorem{thm}{Theorem}[section]
\newtheorem{cor}[thm]{Corollary}
\newtheorem{prop}[thm]{Proposition}
\newtheorem{lem}[thm]{Lemma}
\newtheorem{lemma}[thm]{Lemma}

\theoremstyle{definition}
\newtheorem{defn}[thm]{Definition}
\newtheorem{ex}[thm]{Example}

\newtheorem{ass}[thm]{Assumption}

\numberwithin{equation}{section}

\newcommand{\e}{\varepsilon}

\newcommand{\sdn}[1]{{\color{red}\ttfamily\upshape\small[#1]\color{black}}}

\title{Intermediate disorder for directed polymers with space-time correlations}
\author{Shalin Parekh}
\date{}

\begin{document}

\maketitle
\vspace{-0.2 in} 
\begin{abstract}
    We revisit a result of Hairer-Shen \cite{HS} on polymer-type approximations for the stochastic heat equation with a multiplicative noise (SHE) in $d=1$. We consider a more general class of polymer models with strongly mixing environment in space and time, and we prove convergence to the It\^o solution of the SHE (modulo shear). The environment is not assumed to be Gaussian, nor is it assumed to be white-in-time. Instead of using regularity structures or paracontrolled products, we rely on simpler moment-based characterizations of the SHE to prove the convergence. However, the price to pay is that our topology of convergence is weaker. 
\end{abstract}

\section{Introduction}

A directed polymer in a random environment is a probability measure on paths which re-weights a simple random walk path by some Boltzmann weight that is given by a sum of random variables along the path. These random variables are typically chosen to be independent and identically distributed (IID) at distinct space-time sites. Directed polymers were introduced in a physics paper \cite{HH85} and have since received much attention from the mathematics community for connections to disordered systems, Kardar-Parisi-Zhang universality, and stochastic PDEs. See \cite{Cor12, Qua11, Com17} and related expositions. 

In each spatial dimension, these polymers have a critical temperature that serves as a boundary between weak-disorder (diffusive path behavior) and strong-disorder (superdiffusive path behavior). A lot of important information is contained in the \textit{partition function} of such a path measure. In spatial dimension $d=1$, a paper of \cite{AKQ14a} first proved that under a certain \textit{intermediate-disorder} scaling, this partition function converges to a stochastic PDE called the multiplicative-noise stochastic heat equation. The main observation of \cite{AKQ14a} is that the inverse temperature scales as $O(N^{-1/4})$ in order to observe this intermediate-disorder regime, if time and space are taken to be $O(N)$ and $O(N^{1/2})$ respectively.

That intermediate-disorder result has been further studied and generalized in a variety of works and settings. The particular setting of interest for this paper is the case where the random variables of the environment are not IID but have some weak correlations in both space and time. The main idea is that such a choice of weak \textit{microscopic} space-time correlations should not change the \textit{macroscopic} behavior of the system. This type of problem has already been studied in \cite{Hai13, Gub, HS, HL18, PR19} using regularity structures or paracontrolled products, and in \cite{GT} using more probabilistic methods. Our goal here is to recover and generalize some of those results using a different approach. The non-Gaussian approximation result from \cite{HS} will be the main inspiration, but our central focus will be on polymer models rather than stochastic PDEs. Our main result (Theorem \ref{mr}) focuses on a discrete lattice-based model, but we explain in the next section how our methods also easily apply to semi-discrete or continuous polymer models, including stochastic PDEs whose solutions admit a Feynman-Kac representation, which is indeed the case in \cite{HS}.

Throughout this paper, we are going to fix some probability space $(\Omega, \mathcal F, \mathbb P)$. Let $\omega:= \{\omega_{t,x}\}_{(t,x)\in\mathbb Z^2}$ denote \textbf{the environment}, where $\omega_{t,x}$ are random variables defined on this common probability space $(\Omega, \mathcal F, \mathbb P)$. For any subset $A \subset \mathbb Z^2$ define the $\sigma$-field $\mathcal F_A:= \sigma( \{\omega_{t,x}:(t,x)\in A\})$. We also assume that the probability space admits a $\mathbb Z$-valued random walk $(R(t))_{t\in \mathbb Z_{\ge 0}}$ that is \textit{independent of the environment} $\omega$. In other words, the probability measure splits as $\mathbb P = P^\omega\otimes \mathbf P_{\mathrm{RW}}$, where $P^\omega$ denotes the law of the environment and $\mathbf P_{\mathrm{RW}}$ denotes the law of the random walk. We assume that the $\mathbb Z$-valued random walk $(R(t))_{t\in \mathbb Z_{\ge 0}}$ starts at the origin, is \textit{aperiodic}, and has \textit{bounded jumps of variance 1 and mean 0}.

\begin{ass}\label{ass1} Assume that the environment $\omega$ is \textit{strictly stationary in both variables}. Assume $\omega_{0,0}$ is deterministically bounded, and of mean zero. Also assume that $\omega$ has \textit{finite-range dependence}: there exists $M>0$ such that for any two subsets $A,B \subset \mathbb Z^2$, the $\sigma$-fields $\mathcal F_A,\mathcal F_B$ are independent for $\mathrm{dist}(A,B)>M$, where $\mathrm{dist}(A,B):= \min\{|x-y|:x\in A, y\in B\}.$
\end{ass}
Note that we do not assume time-independence of $\omega$, which is very often used as a simplifying assumption in the literature. Define the \textbf{directed polymer partition function} 
\begin{equation} \label{zn}Z^N_{s,t}(x,y):= \mathbb E \big[e^{N^{-1/4} \sum_{u=0}^{t-s} \omega_{t-u,y+R(u)} }\ind_{\{R(t-s)=x-y\}}\big| \omega \big], 
\end{equation}where $s,t,x,y\in \mathbb Z $ with $s\le t$, and $N\in \mathbb N$. 
The expectation is over the random walk path $(R(t))_{t\in \mathbb Z_{\ge 0}}$, conditional on the environment $\omega$ which we again emphasize is always assumed to be \textit{independent} of the random walk $R$. 
Thus \eqref{zn} can be viewed as a weighted sum over paths from $(s,x)$ to $(t,y)$, with the weight being determined by the environment $\omega$.

\begin{thm}\label{mr}
    Let $\omega=\{\omega_{t,x}\}_{(t,x)\in\mathbb Z^2}$ be an environment satisfying Assumption \ref{ass1}, and let $Z^N$ be defined from $\omega$ as in \eqref{zn}. There exist real constants $c_N, v,$ and $\sigma$ such that the following holds true. For $s,t \in N^{-1}\mathbb Z$ and $x,y\in N^{-1/2}\mathbb Z$, define $$\mathcal Z^N_{s,t}(x,y):=N^{1/2}\cdot e^{-c_N (t-s)}\cdot Z^N_{Ns,Nt}(N^{1/2}x, N^{1/2} y).$$ 
    Then $\mathcal Z^N$ converges in law as $N\to \infty$, as a measure-valued stochastic flow. The limit is given by the It\^o solution of the stochastic PDE 
    \begin{equation}\label{she}\tag{SHE}\partial_t U (t,x)= \tfrac12 \partial_x^2 U(t,x) + v\partial_x U(t,x)+  \sigma U(t,x)\xi(t,x),\;\;\;\;\;\; t\ge 0 ,x\in \mathbb R,\end{equation}
    where $\xi$ is a Gaussian space-time white noise. More precisely, the 
    convergence is in the following integrated sense. Fix $m\in \mathbb N$ and deterministic functions $f_j\in C_c^0(\mathbb R^2)$, for $1\le j \le m$. Consider any sequences $s_j^N<t_j^N$ with $s_j^N,t_j^N\in N^{-1}\mathbb Z$, and assume $s_j^N\to s_j$ and $t^N_j \to t_j$, as $N\to\infty$ for $1\le j \le m$. Then one has the joint convergence in distribution as $N\to \infty$: $$\bigg( N^{-1} \sum_{x,y\in N^{-1/2} \mathbb Z} f_j(x,y) \mathcal Z^N_{s_j^N,t_j^N}(x,y)  \bigg)_{1\le j \le m} \stackrel{(d)}{\longrightarrow}\bigg( \int_{\mathbb R^2} f_j(x,y) U_{s_j,t_j} (x,y) dxdy \bigg)_{1\le j \le m}. $$ 
    Here $U_{s,t}$ on the right side are the propagators for \eqref{she}, meaning that for each $s,x\in \Bbb R$ the function $(t,y)\mapsto U_{s,t}(x,y)$ is a.s the It\^o solution of \eqref{she} started at initial time $t=s$ from Dirac initial condition $y\mapsto \delta_x(y)$, all coupled via the same realization of $\xi$. \end{thm}

In addition to the integrated version above, we will also prove the following version with one fixed endpoint. Fix $m\in \mathbb N$ and deterministic functions $\phi_j\in C_c^0(\mathbb R)$, for $1\le j \le m$. Consider any sequences $s_j^N<t_j^N$ with $s_j^N,t_j^N\in N^{-1}\mathbb Z$, and assume $s_j^N\to s_j$ and $t^N_j \to t_j$, as $N\to\infty$ for $1\le j \le m$. Also consider $y_j^N\in N^{-1/2}\Bbb Z$ such that $y_j^N\to y\in \Bbb R$ as $N\to \infty$. Then one has the joint convergence in distribution as $N\to \infty$: 
\begin{align}\bigg( N^{-1/2} \sum_{x\in N^{-1/2} \mathbb Z} \phi_j(x) \mathcal Z^N_{s_j^N,t_j^N}(x,y_j^N)  \bigg)_{1\le j \le m} \stackrel{(d)}{\longrightarrow}\bigg( \int_{\mathbb R} \phi_j(x) U_{s_j,t_j} (x,y_j) dx \bigg)_{1\le j \le m}. \label{endp}\end{align}

Let us write the expressions for the constants $c_N,v,$ and $\sigma. $ First, we have \begin{equation}\label{sigma}\sigma^2: =\sum_{(t,x) \in \mathbb Z^2} \mathbb E[\omega_{0,0}\omega_{t,x}].\end{equation}
Let $\eta_t:= \omega_{-t,R(t)}$ where (as always) the random walk $R$ is independent of the environment $\omega$, and note that $\eta$ is a strictly stationary sequence in $t$ with finite-range dependence. Then  
\begin{equation}\label{cn}c_N:= \log \mathbb E[ e^{N^{-1/4}\sum_{s=0}^{N} \eta_s}]=\log \sum_{x\in \mathbb Z} \mathbb E[ Z_{0,1}^N(x,0)] .\end{equation}
In Section 3, we will show that $c_N$ can actually be written as $c_2^*N^{1/2} + c_3^* N^{1/4} + c_4^* +O(N^{-1/4})$ for some constants $c_j^*$ which are determined by joint cumulants of order $j$ of the sequence $(\eta_s)_s$, which is similar to the expression for the height shift/renormalization observed in \cite{HS} for a related model. See Lemma \ref{3.1} for the exact expressions of these constants $c_j^*$. Finally, letting $\mathbf f(t,x):=\mathbb E[\omega_{0,0}\omega_{t,x}]$, the shearing constant in Theorem \ref{mr} is given by 
\begin{equation}\label{v}v:=\tfrac12 
   \sum_{k,j\in\mathbb Z} \mathrm{Cov} \big( R(1)-R(0), \mathbf f( k, R(k+j) - R(j))\big).
\end{equation}
The covariance is with respect to a two-sided random walk path $(R(t))_{t\in \mathbb Z}$ distributed as $\mathbf P_{\mathrm{RW}}$. We remark that $v=0$ in many interesting cases. It is certainly zero if the environment $\omega$ is assumed to be independent in time. As pointed out in \cite{HS}, it is also zero if one assumes that $R$ is a symmetric random walk together with the even condition $\mathbf f(t,x) = \mathbf f(t,-x)$, which is clear by substituting $R\to-R$ in the covariance formula for $v$.

Note that \eqref{sigma} and \eqref{v} are actually finite sums by Assumption \ref{ass1}, but we always write such expressions as sums over all $\mathbb Z$. As first observed by \cite{HS}, the appearance of the shearing term $v\partial_xU(t,x)$ in \eqref{she} is perhaps one of the more interesting aspects of the result, since it only appears as an artifact when there is time correlation of the weights $\omega$. In \cite{HS} the authors obtained this term from algebraic renormalization considerations in the theory of regularity structures, but in this paper we will derive it via Girsanov's theorem instead. It is called a shearing term since $(t,x)\mapsto U(t,x-vt)$ solves \eqref{she} with $v=0.$

\subsection{More general models and context}

Here we will discuss easy generalizations of our approach, and overall context and comparisons with the known results. In the main result, we required the underlying space-time lattice to be $\mathbb Z^2$, we required $\omega_{0,0}$ to be deterministically bounded, and we required the random walk to be aperiodic. These conditions may seem overly restrictive, and indeed it is the case that one can consider more general models, which we now explain. 

The aperiodicity assumption on the random walk can of course be relaxed, however one would need to modify the definition \eqref{zn} of $Z^N_{s,t}(x,y)$ appropriately. Rather than defining it for all $(s,t,x,y)\in \Bbb Z^4$ with $s\le t$, one would only define it for those $(x,y)$ lying in some proper sublattice of $\Bbb Z^2$ that depends on $(s,t)$, and zero elsewhere. For example, in the case of nearest-neighbor simple random walk, one should define it only on those $(s,t,x,y)$ such that $t-s$ and $x-y$ have the same parity, and zero elsewhere. Then the rescaled version $\mathcal Z^N$ in Theorem \ref{mr} should have an extra factor of 2 to compensate. Without these periodicity modifications, the convergence result of $\mathcal Z^N$ to \eqref{she} would simply be false.


We cannot relax the finite-range dependence of $\omega$ (e.g. to exponentially fast $\alpha$-mixing or $\phi$-mixing \cite{bradley}) because this is needed to be able to apply the crucial result of \cite{Fer} as done in Lemma \ref{appx} to estimate the remainder. The boundedness of the weights $\omega$ is used for the same reason in the same lemma, but it can be relaxed to e.g. Gaussian tails using a truncation argument.

Generalization to non-lattice models is easily allowable with our methods, without any substantial modification needed in the proofs. For example, we can look at ``semi-discrete" polymers, where the underlying space-time lattice is $(t,x) \in \mathbb Z\times \mathbb R$, and where $\omega$ satisfies a similar mixing condition as in Assumption \ref{ass1} plus path continuity. In this case, the underlying random walk measure $\mathbf P_{\mathrm{RW}}$ should have an increment law that has a continuous and compactly supported density with respect to Lebesgue measure, mean zero, and variance one. The definition of $Z^N$ generalizes easily to this case, and one still has Theorem \ref{mr}, with the same proof and all constants being essentially the same, except that sums over $x\in \mathbb Z$ would now be replaced by integrals over $x\in\mathbb R$ in the expressions \eqref{sigma}, \eqref{cn}, and \eqref{v}. One can also consider the opposite choice of ``semi-discrete" polymer, where $(t,x)\in \mathbb R\times \mathbb Z$, and the random walk is a continuous-time process on a discrete lattice. In this case, the definition of $Z^N$ in \eqref{zn} should have a time-integral as opposed to a sum, but our methods still apply and give similar expressions for the constants, with time-sums replaced by time-integrals. Likewise, our approach generalizes to the case where the underlying space-time ``lattice" is $(t,x) \in \mathbb R^2,$ but all the sums over $t,x\in \mathbb Z$ would now be replaced by integrals, in both the definition of $Z^N$ and in all of the constants $c_N,v,\sigma$. The path measure must be a standard Brownian motion in this case, and $\omega$ must be a.s. continuous on $\mathbb R^2$, but the result and the proof still apply for such a model. In particular, \textit{our methods do apply to the models considered in \cite{HS, GT}} since those SPDEs have Feynman-Kac representations in terms of Brownian-weighted polymers, very similar to the expression \eqref{zn} for $Z^N$ above.


Let us therefore explain some of the main differences of our approach with those of \cite{HS}, which is the inspiration and the closest existing result in the literature to our Theorem \ref{mr}, since they also consider nontrivial time correlations and non-Gaussian noises. The paper of \cite{HS} builds on the theory of regularity structures as developed in \cite{Hai13, Hai14}, working at the level of the logarithm of \eqref{she}: the so-called Kardar-Parisi-Zhang (KPZ) equation. Theorem 1.3 of their paper proves a result very similar to our Theorem \ref{mr}, where they also find renormalization constants $c_N$ and $v$ as above. However, there are a few differences that we now explain. Firstly, their spatial domain is a circle, as opposed to the full real line or integer line. This is necessary for them to be able to apply the regularity structure machinery, since the KPZ equation has not been solved on $\mathbb R$ with regularity structures yet (though see \cite{HL18, PR19}). Another restriction of their approach is that they require the initial data to be quite regular. Furthermore, their approach is confined to \textit{continuum approximations} of \eqref{she} that are already given by stochastic PDEs, whereas we can consider \textit{lattice-based approximations} as in Theorem \ref{mr}, or even semi-discrete approximations as explained in the previous paragraph. It may be possible to generalize the approach of \cite{HS} to lattice approximations using discretization of regularity structures \cite{disc1,disc2}, but the effort would be longer and more technical. Despite these slight drawbacks, the result of \cite{HS} does have several major benefits over our methods. The first major advantage of \cite{HS} over our approach is that they are able to prove \textit{tightness} in a much stronger topology, namely the space of H\"older continuous functions in space-time. In our approach, we can obtain integrated results like Theorem \ref{mr} and pointwise results like \eqref{endp}, but there is no guarantee of continuity. Another major benefit of \cite{HS} is that their approach generalizes to other limiting stochastic PDEs which do not necessarily linearize under some simple transform such as Hopf-Cole (see \cite{HP15, CS, HQ18}), whereas we are restricted to linear equations such as \eqref{she}. Still another benefit of the \cite{HS} approach is that they can obtain \textit{joint} convergence of the environment $\omega$ and the field $\mathcal Z^N$ to the space time white noise $\xi$ and the solution of \eqref{she} driven by that same noise $\xi$, which is quite natural to ask. Finally, we require boundedness of the field $\omega$, whereas \cite{HS} states as a remark that they only need some high but finite number of moments.

Next let us explain the differences of our approach with that of \cite{GT}, which can also consider certain Gaussian polymers with nontrivial time-correlations of the noise, and gives a more direct probabilistic proof of convergence on the full line $\mathbb R$ that does not rely on regularity structures \cite{Hai14} or paracontrolled products \cite{GIP15}. To avoid the use of these theories, \cite{GT} heavily leverages the \textit{Gaussian structure} of the noise, and in particular uses Malliavin calculus and the Clark-Ocone formula to prove the convergence. This is a very strong restriction. Here, we make no use of Gaussianity, and instead we allow for very general fields of noises as long as the strong mixing condition of Assumption \ref{ass1} is satisfied. That being said, the methods of \cite{GT} for the Gaussian case should be robust enough to prove the tightness in a space of continuous functions, and in fact they do explicitly prove convergence in law for each individual value of $(s,t,x,y)$, jointly with the noise, which is still a stronger statement than our integrated statement in Theorem \ref{mr}.


Therefore, letting $\mathbb R^4_\uparrow:=\{(s,t,x,y): s<t\}$, the main open problem left unsolved here is to prove the \textit{tightness} of the field $Z^N$ in $C(\mathbb R^4_{\uparrow})$, with a succinct and direct probabilistic approach. In the case that the weights are \textit{independent-in-time}, this turns out to be fairly straightforward by writing out the Duhamel formula for the discrete system, thus obtaining a martingale-driven ``discrete SPDE," and then using the Burkholder-Davis-Gundy (BDG) inequality in an optimal way. In the case that the weights are not independent-in-time, the BDG inequality no longer applies since one loses the Markov property and thus the martingality. It is unclear if this problem can be easily fixed, or if it requires a more intricate analysis. This may be pursued later, in a future work. 

\subsection{Method of proof.}

The proof mostly elementary, and will rely on Taylor expansions of the cumulant generating function for the stationary sequence $\eta_t := \omega_{-t,R(t)}$ introduced just after Theorem \ref{mr}. Ultimately these cumulant expansions will yield formulas amenable to proving \textit{convergence of moments} of $Z^N$ from \eqref{zn} to those of $U$ from \eqref{she}, after which we will use a moment-based characterization of \eqref{she} inspired by work of \cite{Ts24} to complete the proof of Theorem \ref{mr}. Our proof is not similar to the proofs in \cite{HS} or in \cite{GT}, which had different ways of deriving all of the constants. In \cite{HS} the nontrivial renormalization constants $c_N$ and $v$ were interpreted as the correct ones needed to lift the noises in such a way as to obtain nontrivial ``probabilistic lifts" of the Gaussian noise in the $N\to\infty$ limit, while in \cite{GT} they were derived through an exact analysis of certain Brownian functionals. 
\\
\\
\textbf{Outline. } In Section 2, we rewrite the constant $c_N$ using a Taylor expansion, which allows us to estimate first moments of the process $Z_{s,t}^N(x,y)$. In Section 3, we calculate the higher moments in terms of expectations of certain exponential functionals of random walks. In Section 4, we prove some limit theorems for various processes that appear in the expressions for those higher moments. Finally in Section 5, we prove the main result by using the higher moment convergence results.
\\
\\
\textbf{Acknowledgement.} I thank Yu Gu for discussion and references. I acknowledge support by the NSF MSPRF (DMS-2401884).

\section{Renormalization coefficient and cumulant expansion}

Recall the renormalization constant $c_N:=\log \mathbb E[ e^{N^{-1/4}\sum_{s=0}^{N} \eta_s}]$ from \eqref{cn} in the introduction. These constants $c_N$ will be instrumental in the analysis, but they are hard to work with directly, thus we prove some nice properties about them in this section, by expanding some cumulants. First we establish some notations. 

\begin{defn}
    For a random variable $X$, let $\kappa_n(X)$ denote the $n^{th}$ cumulant of $X$. Likewise for random variables $X_1,...,X_n$ all defined on the same probability space, let $\kappa_{(n)}(X_1,...,X_n)$ denote the joint cumulant of the random variables $X_1,...,X_n$. 
\end{defn}

In other words, $$\kappa_{(n)}(X_1,...,X_n) = \partial_{t_1}\cdots \partial_{t_n} \big|_{(t_1,...,t_n)=(0,...,0)} \log \mathbb E[e^{t_1X_1+..+t_nX_n}],$$ and $\kappa_n(X) = \kappa_{(n)}(X,...,X)$. We remark that $\kappa_{(n)}$ is linear in each coordinate, and symmetric under interchange of coordinates, which will be very important. For example $\kappa_{(1)}(X) = \mathbb E[X],$ $\kappa_{(2)}(X,Y) = \mathrm{Cov}(X,Y),$ $\kappa_{(3)}(X,Y,Z) = \mathbb E[XYZ] - \mathbb E[X] \mathbb E[YZ] -\mathbb E[Y] \mathbb E[XZ] - \mathbb E[Z] \mathbb E[XY] + 2\mathbb E[X]\mathbb E[Y]\mathbb E[Z]$, and so on. Another important property is that if $X_1$ is independent of $(X_2,...,X_n)$ then we have that $\kappa_{(n)} (X_1,...,X_n)=0$, in other words, \textit{just one} of the $X_i$ needs to be independent of the rest of the collection for the joint cumulant to vanish. This is a crucial observation going forward.

\begin{defn}\label{fn} Recall the stationary sequence $\eta_t:= \omega_{-t,R(t)}\;\; (t\in \mathbb Z)$ from the introduction, where $\omega$ is the environment and $R$ is an independent random walk. 
Let $$f_{\mathrm{cgf}}(N,\lambda):= \log \mathbb E[ e^{\lambda \sum_{s=0}^N \eta_s }].$$
\end{defn}

As a cumulant generating function $f_{\mathrm{cgf}}(N,\lambda)$ can be expanded in $\lambda$ as 
$$f_{\mathrm{cgf}}(N,\lambda) = \sum_{n=1}^\infty \frac1{n!}\kappa_n(\eta_1+...+\eta_N) \lambda^k $$ where $\kappa_n$ are the order-$n$ cumulants. For example, by the stationarity and mean-zero property of the fast-mixing sequence $\eta$, we can write out the first four cumulants in the expansion as follows:
\begin{align*}
    \kappa_1( \eta_1+...+\eta_N)&= \mathbb E [ \eta_1+...+\eta_N]=0,\\\kappa_2( \eta_1+...+\eta_N) &= \sum_{i,j=1}^N \mathrm{Cov}(\eta_i,\eta_j)  = \sum_{i=-N}^N(N-|i|) \mathbb E[ \eta_0\eta_i]\\
    \kappa_3(\eta_1+...+\eta_N) &= \sum_{i,j,k=1}^N \kappa_{(3)}( \eta_i,\eta_j,\eta_k) = \sum_{i,j,k=1}^N \mathbb E[\eta_i\eta_j\eta_k]\\ \kappa_4( \eta_1+...+\eta_N) &= \sum_{i,j,k,\ell=1}^N \kappa_{(4)}(\eta_i,\eta_j,\eta_k,\eta_\ell) \\&= \sum_{i,j,k,\ell=1}^N \mathbb E[ \eta_i\eta_j\eta_k\eta_\ell] - \mathbb E[\eta_i\eta_j]\mathbb E[\eta_k\eta_\ell] -\mathbb E[ \eta_i\eta_\ell]\mathbb E[\eta_j\eta_k]-\mathbb E[ \eta_i\eta_k ]\mathbb E[\eta_j\eta_\ell].
\end{align*}

\begin{lem}\label{3.1}
    As $N\to\infty$ we have the following convergences 
    \begin{align*}
        \frac1N \kappa_2(\eta_1+...+\eta_N) &\to \sum_{s=-\infty}^\infty \mathbb E[\eta_0\eta_s] =:2!c_2^*, \\\frac1N \kappa_3(\eta_1+...+\eta_N) &\to\sum_{s,t=-\infty}^\infty \mathbb E[\eta_0\eta_s\eta_t ] =:3!c^*_3,\\ \frac1N \kappa_4(\eta_1+...+\eta_N) &\to \sum_{s,t,u=-\infty}^\infty \mathbb E[ \eta_0\eta_s\eta_t\eta_u] - \mathbb E[\eta_0\eta_s]\mathbb E[\eta_t\eta_u] -\mathbb E[ \eta_0\eta_u]\mathbb E[\eta_s\eta_t]-\mathbb E[ \eta_0\eta_t ]\mathbb E[\eta_s\eta_u] =:4!c^*_4.
    \end{align*}
    Moreover, these convergences occur at a rate of $O(1/N).$
\end{lem}

\begin{proof}
    We only sketch the proof for the second cumulant, with the higher cumulants being similar. Using the above expressions, we have that $$\frac1N \kappa_2(\eta_1+...+\eta_N) = \sum_{i=-N}^N (1-\frac{|i|}N) \mathbb E[\eta_0\eta_i],$$ thus we find that $$|c_2^* - \frac1N \kappa_2(\eta_1+...+\eta_N) | \leq \frac1N\sum_{i=-N}^N |i| \mathbb |E[\eta_0\eta_i]| \le C/N,$$ where we use $\mathbb E[\eta_0\eta_i] =0 
    $ for sufficiently large $|i|$, which in turn is thanks to the strong mixing and mean-zero property of the sequence $\eta$ (see Assumption \ref{ass1}). 
    As mentioned above, the proof for higher cumulants is similar.
\end{proof}

\begin{lemma}\label{appx}
    We have that $|\kappa_n(\eta_1+...+\eta_s)| \leq C^nn! s$ where $C$ is independent of $n,s \in \mathbb N$.
\end{lemma}

\begin{proof} Despite the seemingly innocuous nature of the bound, this is where we need the restrictive assumptions of boundedness and finite-range-dependence on the weights $\omega$ in Assumption \ref{ass1}. Under these assumptions, the same boundedness and finite-range dependence holds true for the sequence $\eta$, and the result thus follows precisely from Theorem 9.1.7 in \cite{Fer}. We remark that the latter is a fairly deep result and requires the analysis of certain very precise and exact cancellations in the cumulants, using Tutte polynomials and combinatorial structures that goes beyond the brutal absolute value bounds that can be obtained just using the definition of a cumulant.
\end{proof}

The purpose of the above lemma will be to control the remainder (terms of order $\ge 5$) when Taylor expanding the cumulant generating function. While the bound in \cite{Fer} is the strongest one that we could find in the literature, there is some previous work devoted to bounding cumulants, see e.g. \cite{Jan, Rao,Dou1,Dou2}. Although we will not use the following fact, we also remark that the bound in the previous lemma is equivalent to the zero-freeness condition needed to apply Bryc's CLT to the sequence $X_t:=t^{-1} (\eta_1+...+\eta_t)$, see \cite[Eq. (1.2)]{Bry3} as well as some related work of \cite{Bry1,Bry2}.

The above observations now lead to the following theorem. 

\begin{thm}\label{ttc}
    Let $c_N$ be as in the introduction, and $c_j^*$ as in Lemma \ref{3.1}. We have that $|c_N - c_2^* N^{1/2} -c_3^* N^{1/4} - c_4^*| \leq CN^{-1/4}.$ More generally, let $f_{\mathrm{cgf}}$ be as in Definition \ref{fn}. For any fixed time horizon $T>0$, we have the bound $$\sup_{t\in (N^{-1}\mathbb Z)\cap [0,T] } |f_{\mathrm{cgf}}(Nt,N^{-1/4}) - (c_2^* N^{1/2} +c_3^* N^{1/4} + c_4^*)t|\leq C N^{-1/4}.$$
\end{thm}

\begin{proof}
    Writing $c_N:= f_{\mathrm{cgf}}(N,N^{-1/4}) $ shows why the second bound is more general (make $t=1$). To prove the second bound, use 
    $$\sup_{t\in (N^{-1}\mathbb Z)\cap [0,T] } \sup_{|\lambda|\leq \lambda_0 }|f_{\mathrm{cgf}}(Nt,\lambda ) - \sum_{k=1}^4 \frac1{k!} \kappa_k(\eta_1+...+\eta_{Nt}) \lambda^k |\leq C \sum_{n=5}^\infty \frac{\lambda_0^n }{n!} |\kappa_n(\eta_1+...+\eta_{Nt})|.$$ For $j=1,2,3,4$, Lemma \ref{3.1} shows that $|\kappa_j(\eta_1+...+\eta_{Nt}) - c_j^*Nt| \leq C'$ for some absolute constant $C'$. Also, from Proposition \ref{appx}, one has that $|\kappa_n(\eta_1+...+\eta_{s})|\leq (C^nn! )s$ with $C$ independent of $n,s$. Thus we get $$\sup_{t\in (N^{-1}\mathbb Z)\cap [0,T] } \sup_{|\lambda|\leq \lambda_0 }|f_\mathrm{cgf}(Nt,\lambda ) - \sum_{k=1}^4 \frac1{k!} Nc_k^*t  \lambda^k | - C'\sum_{k=1}^4|\lambda|^k\leq C \lambda_0^5\cdot N.$$
    Making $\lambda=\lambda_0=N^{-1/4}$ clearly yields the claim.
\end{proof}

The above can be viewed as precise first-moment asymptotics for the field $Z^N$ from \eqref{zn}. For example by \eqref{cn} we easily obtain uniformly over $t\in (N^{-1}\mathbb Z)\cap [0,T]$ that
$$|e^{-c_Nt} \sum_{x\in \mathbb Z} \mathbb E[Z^N_{0,t}(x,0)] - 1| \leq CN^{-1/4},$$ just by taking exponential in the previous theorem. The next section will be devoted to deriving higher moment formulas, with more general parameters allowed. This higher moment analysis will be the key to proving Theorem \ref{mr}.

\section{Calculating higher moments}

Fix $k\in\mathbb N$. For $s\in \mathbb Z_{\ge 0}$ and $\vec y\in \mathbb Z^k$ and $\phi:\mathbb R\to\mathbb R$ (bounded, measurable) let us define $$z_\lambda (k,\phi; s,\vec y):= \mathbb E \bigg[ e^{\lambda \sum_{j=1}^k  \sum_{r=0}^{s} \omega_{s-r, y_j+R^j(r)}}\prod_{j=1}^n \phi(R^j(s)+y_j)\bigg]. $$
Here the expectation is taken \textbf{both} over the weights $\omega=\{\omega_{t,x}\}_{(t,x)\in \Bbb Z^2}$ and the $k$ independent random walk paths $(R^1,...,R^k)$ started at the origin, which (as always) are assumed to be independent. The relevance of $z_\lambda$ is as follows.

\begin{prop}\label{0} Let $k\in \mathbb N,$ let $y_1,...,y_k\in N^{-1/2}\mathbb Z$, and let $\phi$ be a continuous and bounded function on $\mathbb R$. With $\mathcal Z_N$ as defined in Theorem \ref{mr}, we have that $$\sum_{x_1,...,x_k\in N^{-1/2}\mathbb Z}\mathbb E[ \mathcal Z^N_{0,t}(x_1,y_1)\cdots \mathcal Z^N_{0,t}(x_k,y_k)] \prod_{j=1}^k \phi(x_j) = e^{-c_Nt} \cdot z_{N^{-1/4}} ( k,\phi(N^{-1/2}\bullet) ; Nt, N^{1/2} \vec y).$$
\end{prop}

The proof is clear just by definitions. 
We will now use the tower property of conditional expectation to write $$z_{\lambda}(k,\phi ;s,y) = \mathbf E_{\mathrm{RW}^{\otimes k}} \bigg[ \mathbb E \big[ e^{\lambda \sum_{j=1}^k  \sum_{r=0}^{s} \omega_{s-r, y_j+R^j(r)}}|R^1,...,R^k\big] \cdot \prod_{j=1}^k \phi(R^j(s)+y_j)\bigg],$$
where the inner expectation is only over the weights $\omega$ (conditioned on the paths $R^j$), and the outer expectation is over $k$ independent random walk paths. Had we assumed that the weights were independent \textit{in time}, calculating the inner expectation would be pretty simple. However, it becomes much more interesting with our assumption of time correlations. Instead, rewrite it as 
\begin{equation}\label{IMP}z_{\lambda}(k,\phi ;s,y) = \mathbf E_{\mathrm{RW}^{\otimes k}} \bigg[ e^{\log\mathbb E \big[ e^{\lambda \sum_{j=1}^k  \sum_{r=0}^{s} \omega_{s-r, y_j+R^j(r)}}|R^1,...,R^k\big]} \cdot \prod_{j=1}^k \phi(R^j(s)+y_j)\bigg].
\end{equation}
Now expand the cumulant generating function in the variable $\lambda$: 
\begin{equation}\label{ex}\log\mathbb E \big[ e^{\lambda \sum_{j=1}^k  \sum_{r=0}^{s} \omega_{s-r, y_j+R^j(r)}}|R^1,...,R^k\big] = \sum_{n=1}^\infty \frac1{n!} \kappa_n \bigg( \sum_{j=1}^k  \sum_{r=0}^{s} \omega_{s-r, y_j+R^j(r)} \bigg| R^1,...,R^k \bigg)\lambda^n.
\end{equation}
Given deterministic paths $\gamma^1,...,\gamma^k: \mathbb Z\to\mathbb Z$ we need to be able to compute or estimate the quantities 
\begin{equation}\label{imp2}\frac1{n!}\kappa_n\bigg( \sum_{j=1}^k  \sum_{r=0}^{s} \omega_{s-r, y_j+\gamma^j(r)}\bigg) = \frac1{n!}\sum_{j_1,...,j_n=1}^k \sum_{r_1,...,r_n=0}^s \kappa_{(n)} \big( \omega_{s-r_1 , y_{j_1}+\gamma^{j_1}(r_1)},...,\omega_{s-r_n , y_{j_n}+\gamma^{j_n}(r_n)}\big).
\end{equation}
Let us consider separate cases of $n$ in order to analyze \eqref{imp2}.
\\
\\
\textbf{The case $n=1$: } The $n=1$ sum in \eqref{imp2} vanishes by the mean-zero assumption on the environment $\omega$: $$\sum_{j=1}^k \sum_{s=0}^r \mathbb E[ \omega_{s-r, y_j + \gamma^j(r)}]=0.$$
\textbf{The case $n=2$: } This will be the most important case. Write $\mathbf f(t,x):= \mathbb E[\omega_{0,0}\omega_{t,x}]$. Note that $\kappa_{(2)} (A,B) = \mathrm{Cov}(A,B)$, thus for $n=2$ we can write \eqref{imp2} as 
\begin{align*}\sum_{j_1,j_2=1}^k \sum_{r_1,r_2=1}^s Cov \big( \omega_{s-r_1 , y_{j_1}+\gamma^{j_1}(r_1)},&\omega_{s-r_2 , y_{j_2}+\gamma^{j_2}(r_2)}\big) \\&=\sum_{j_1,j_2=1}^k \sum_{r_1,r_2=1}^s \mathbf f\big( r_1-r_2, y_{j_1}-y_{j_2} + \gamma^{j_1} (r_1) - \gamma^{j_2}(r_2)\big).\end{align*}
Writing $\boldsymbol\gamma:= (\gamma^1,...,\gamma^k)$, we are going to split the above sum into two parts. The first part $\mathbf{Diag}_s(\boldsymbol\gamma)$ will consist of appropriately normalized ``diagonal terms" with $j_1=j_2$, and the second part $\mathbf{Off}_s(\boldsymbol\gamma, \vec y)$ will consist of ``off-diagonal" terms. More precisely, we define the quantities 
$$ \mathbf{Diag}_s(\boldsymbol\gamma):= \frac12 \sum_{j=1}^k \sum_{r_1,r_2=1}^s \mathbf f(r_1-r_2, \gamma^j(r_1) - \gamma^j(r_2)) \;\;\;\;- \;\;\;\; sc_2^*,$$ where $c_2^*$ is as in Lemma \ref{3.1}, and
$$\mathbf{Off}_s(\boldsymbol\gamma, \vec y):= \frac12 \sum_{j_1\ne j_2} \sum_{r_1,r_2=1}^s \mathbf f\big( r_1-r_2, y_{j_1}-y_{j_2} + \gamma^{j_1} (r_1) - \gamma^{j_2}(r_2)\big). $$
Here $s\in \mathbb Z_{\ge 0}$ and $\boldsymbol\gamma$ can be any $k$-tuple of paths $\mathbb Z\to \mathbb Z$. Since $n=2$, take note of the factor of $\lambda^2$ in the Taylor expansion \eqref{ex}, with $\lambda:=N^{-1/4}$. 

Before moving onwards, let us remark on the critical importance of the above two path functionals $\mathbf{Off}$ and $\mathbf{Diag}$. As one might be able to guess, under the scaling of Theorem \ref{mr}, in the exponential of \eqref{IMP}, this yields terms of the form $N^{-1/2} \mathbf{Off}_{Nt}(\mathbf R, N^{1/2} \vec y) + N^{-1/2} \mathbf{Diag}_{Nt} (\mathbf R)$, where $\mathbf R=(R^1,...,R^k)$ is distributed as $k$ independent random walks. Each of these will contribute differently to the limit. 

Specifically, the diagonal part $N^{-1/2} \mathbf{Diag}_{Nt} (\mathbf R)$ will converge to some Brownian motion that will yield the \textbf{Girsanov tilt} which will be responsible for the shearing constant $v$ in Theorem \ref{mr}. This Brownian motion will be correlated to the one given by the limit of $N^{-1/2} \mathbf R(Nt)$, but not strictly equal to it. See Step 3 in the proof of Theorem \ref{b} just below.

Meanwhile, the off-diagonal part $N^{-1/2} \mathbf{Off}_{Nt}(\mathbf R, N^{1/2} \vec y)$ will converge to \textbf{Brownian local times}, where now the Brownian motion is simply the one given by the limit of $N^{-1/2} \mathbf R(Nt)$. See Step 2 in the proof of Theorem \ref{b} just below.

Each of these contributions $\mathbf{Off}_s$ and $\mathbf{Diag}_s$ will be very important for the scaling limit. 
\\
\\
\textbf{The cases $n=3$ and $n=4$: } Since $n=3$ or $n=4$, take note of the factor of $\lambda^3$ or $\lambda^4$ in the Taylor expansion, with $\lambda:=N^{-1/4}$. This case will therefore not be important in the limit, since $N^{-3/4}$ or $N^{-1}$ is much less than the threshold of $N^{-1/2}$ that is needed to be meaningful in the limit (as noted in the case $n=2$). We will need another separation argument to show this. Specifically, we will separate \eqref{imp2} into diagonal terms and off-diagonal terms again.

Define the following constant, whose significance will become clear in the next section.
\begin{equation}\label{gamma}\gamma^2 := \tfrac14 \sum_{k_1,k_2,j_1,j_2\in\mathbb Z} \mathrm{Cov} \big( \mathbf f(k_1, R(j_1+k_1), \mathbf f(k_2, R(j_2+k_2))\big).
\end{equation}
For $\lambda>0$, $\boldsymbol\gamma: \mathbb Z\to\mathbb Z^k$, and $\vec y \in \mathbb Z^k$, define the ``diagonal parts" by $$\mathbf{Err}^1_s(\lambda;\boldsymbol\gamma;\vec y):=\lambda^3 \sum_{j=1}^k \bigg[\frac1{3!}\sum_{r_1,...,r_3=0}^s \kappa_{(3)} \big( \omega_{s-r_1 , y_{j}+\gamma^{j}(r_1)},...,\omega_{s-r_3 , y_{j}+\gamma^{j}(r_3)}\big) \;\;\;\;-\;\;\;\; sc_3^*\bigg]$$
$$\mathbf{Err}^2_s(\lambda;\boldsymbol\gamma;\vec y):=\lambda^4\sum_{j=1}^k \bigg[\frac1{4!}\sum_{r_1,...,r_4=0}^s \kappa_{(4)} \big( \omega_{s-r_1 , y_{j}+\gamma^{j}(r_1)},...,\omega_{s-r_4 , y_{j}+\gamma^{j}(r_4)}\big) \;\;\;\; - \;\;\;\; s \big( c_4^* - \tfrac12 \gamma^2)\bigg].$$
Here $c_j^*$ are the constants from Lemma \ref{3.1}. Finally define the ``off-diagonal parts" by $$\mathbf{Err}^3_s(\lambda;\boldsymbol\gamma;\vec y):=\frac{\lambda^3}{3!}\sum_{\substack{j_1,...,j_3\\ j_i\;\mathrm{not\; all \;equal}}}^k \sum_{r_1,...,r_3=0}^s \kappa_{(3)} \big( \omega_{s-r_1 , y_{j_1}+\gamma^{j_1}(r_1)},...,\omega_{s-r_3 , y_{j_3}+\gamma^{j_3}(r_3)}\big),$$
$$\mathbf{Err}^4_s(\lambda;\boldsymbol\gamma;\vec y):=\frac{\lambda^4}{4!}\sum_{\substack{j_1,...,j_4\\ j_i\;\mathrm{not\; all \;equal}}}^k \sum_{r_1,...,r_4=0}^s \kappa_{(4)} \big( \omega_{s-r_1 , y_{j_1}+\gamma^{j_1}(r_1)},...,\omega_{s-r_4 , y_{j_4}+\gamma^{j_4}(r_4)}\big),$$where in all expressions $s\in \mathbb Z_{\ge 0}.$ We labeled all of these terms as ``Err" because they represent error terms that will be shown to be irrelevant in the scaling limit.
\\
\\
\textbf{The case $n\ge 5$: } It turns out that we can truncate the infinite sum in \eqref{ex} at this point, without needing to separate into diagonal and off-diagonal cases. The reason is very similar to the proof of Theorem \ref{ttc}, and will be elaborated later (see Step 6 in the proof of Theorem \ref{b} below)
. We thus define
$$\mathbf{Err}^5_s(\lambda;\boldsymbol\gamma;\vec y):=\sum_{n=5}^\infty \frac{\lambda^n}{n!} \sum_{j_1,...,j_n=1}^k \sum_{r_1,...,r_n=0}^s \kappa_{(n)} \big( \omega_{s-r_1 , y_{j_1}+\gamma^{j_1}(r_1)},...,\omega_{s-r_n , y_{j_n}+\gamma^{j_n}(r_n)}\big).$$


Summarizing this whole discussion and applying Proposition \ref{0}, we have proved the following representation for the joint moments of the field $\mathcal Z^N$ appearing in Theorem \ref{mr}: 

\begin{prop}[Moment formula]\label{a}
    Let $k\in \mathbb N,$ let $y_1,...,y_k\in N^{-1/2}\mathbb Z$, and let $\phi$ be a smooth bounded function on $\mathbb R$. For $t\in N^{-1}\mathbb Z_{\ge 0}$, we have that 
    \begin{align*}&N^{-k/2} \sum_{x_1,...,x_k\in N^{-1/2}\mathbb Z}\mathbb E[ \mathcal Z^N_{0,t}(x_1,y_1)\cdots \mathcal Z^N_{0,t}(x_k,y_k)] \prod_{j=1}^k \phi(x_j)\\ &= \mathbf E_{\mathrm{RW}^{\otimes k}}\bigg[ e^{N^{-1/2} \mathbf{Off}_{Nt}(\mathbf R, N^{1/2} \vec y) + N^{-1/2} \mathbf{Diag}_{Nt} (\mathbf R)-\frac{\gamma^2kt}2 + \sum_{a=0}^5 \mathbf{Err}^a_{Nt} (N^{-1/4};\mathbf R, N^{1/2}\vec y) } \cdot \prod_{j=1}^k \phi(y_j+N^{-1/2}R^j(Nt)) \bigg],
    \end{align*}
    where the expectation is over $k$ independent random walk paths $(R^1,...,R^k)=\mathbf R,$ started from the origin. Here $\mathbf{Off}_s$, $\mathbf{Diag}_s$, and $\mathbf{Err}_s$ were all defined in the discussion following \eqref{imp2} above. 
\end{prop}

\section{Limit theorem and exponential moment bounds for all of the processes}

To take the $N\to \infty$ limit of the expectation on the right side in Proposition \ref{a}, it is clear that we will need a \textit{joint} limit theorem for all of the different processes appearing in the expression. We now pursue this, first proving two lemmas before obtaining the full process-level convergence result in Theorem \ref{b}.

\begin{lem}[Limit theorem for additive functionals of random walks on $\mathbb Z$] \label{addfl} Consider two independent random walks $(R^j(t))_{t\in \mathbb Z_{\ge 0}, j=1,2}$ on $\mathbb Z$ of mean zero and increments of variance $1$. Assume they are aperiodic. Let $f:\mathbb Z\to \mathbb R_{\ge 0}$ be of finite support. Consider a deterministic sequence $a_N\in \mathbb Z$ such that $N^{-1/2} a_N\to a\in \mathbb R$. Then for any fixed $k\in \mathbb Z$ and $T>0$, the triple of processes given by $$N^{-1/2} \big( R^1_{Nt}, R^2_{Nt}, \sum_{r=0}^{Nt} f(a_N+R^1(r) - R^2(r-k))\big)_{t\in [0,T]} $$ converges in law as $N\to\infty$, to $(X,Y,\|f\|_{\ell^1(\mathbb Z)} \cdot L_0^{X-Y+a}). $ The topology is that of $C([0,T],\Bbb R^3)$.
\end{lem}

\begin{proof} 
    Marginally, it is clear from Donsker's theorem that the first two coordinates are tight in $C[0,T]$ and converge in law to a pair of independent Brownian motions $(X,Y).$ To prove the tightness of the third coordinate, first notice by the local CLT that 
    \begin{equation}\label{e1}\sup_{x\in \mathbb Z} \mathbb E [ f(x+ R^1(r))] \leq C\| f\|_{\ell^1} r^{-1/2} ,
    \end{equation}
    with $C$ independent of $r\in\mathbb Z_{\ge 0}$ and of $f\in \ell^1(\mathbb Z)$. By the Markov property of $R^1$, we can iterate this bound, and we will obtain that \begin{equation}\label{e2}\sup_{x_1,...,x_k\in \mathbb Z} \mathbb E\bigg[ \prod_{a=1}^m f(x_a + R^1(r_a)) \bigg] \leq \| f\|_{\ell^1}^m \prod_{a=1}^m 1\wedge (r_a- r_{a-1})^{-1/2} 
    \end{equation}whenever $r_1\le ... \le r_m.$ From here we can easily obtain that for any even integer $p$ one has 
    \begin{equation}\label{trunc}\sup_{x\in \mathbb Z} \mathbb E \bigg[ \bigg| \sum_{r=s}^{t} f(x + R^1(r)) \bigg|^m \bigg] ^{1/m} \leq C\| f\|_{\ell^1} |t-s|^{1/2}, 
    \end{equation}
    which easily yields tightness of the third coordinate in $C[0,T],$ since $R^2$ is independent and can be conditioned out as if it were deterministic (e.g. absorbed into the variable $x$ in the above bound). Now we need to identify the limit point. Firstly, by replacing $a_N$ by $a_N+ R^2(r) - R^2(r-k)$, we can reduce the claim to the case where $k=0$, since this operation does not change the limit of $N^{-1/2} a_N$.

    Let us first prove convergence of the first moments. Fix some $t\ge 0$, and let $\pi^N(g):= N^{-1/2} \sum_{s=0}^{Nt} \mathbb E[ g(a_N + R^1(r) - R^2(r))].$ This sequence $\pi^N$ is a sequence of measures on $\mathbb Z$, which has subsequential limits on finite subsets by \eqref{e1}, thus subsequential limits on all $\Bbb Z$ by a diagonal argument (in the topology of integration against finitely supported functions). Any subsequential limit $\pi$ must be an invariant measure for the Markov process $R^1-R^2$, since a direct ``Krylov-Bogoliubov" type argument gives $\pi P=\pi$. Thus by the aperiodicity assumption, any subsequential limit must be a scalar multiple of counting measure on $\mathbb Z.$ If we can show the existence of some $g_0$ such that $\lim_{N\to \infty} \pi^N(g_0)$ exists this would uniquely pin down the constant, thus showing that $\pi^N$ itself converges to the same multiple of counting measure. 

    Let us henceforth abbreviate $Y(r):= R^1(r) - R^2(r),$ which is a symmetric random walk on $\mathbb Z$ with increment distribution having some probability function $p_Y (y) := \mathbb P(Y(1) = y)$, in other words, the one-step probability transition for the Markov chain $Y$. Define the operator $Pg(x) := \sum_{y\in \mathbb Z} p_Y(x-y)g(y) .$ For all $g:\mathbb Z\to \mathbb Z$ such that $Pg$ is well-defined and bounded, one easily checks directly that the process $$M_g(r):= g(Y(r)) - \sum_{s=0}^{r-1} (P-\mathrm{Id})g(Y(s))$$
    is a martingale for the filtration generated by the process $Y$ (the so-called Dynkin martingale). Letting $g_0:= (P-\mathrm{Id}) u_0$ where $u_0(x):= |x|$, we thus see that $\pi^N(g_0) = \mathbb E[ u_0 ( N^{-1/2} Y_{Nt} + N^{-1/2} a_N) ] - u_0(N^{-1/2} a_N),$ which by Donsker's principle (and the assumptions on $a_N$) converges as $N\to \infty$ to the quantity $\mathbf E_{\mathrm{BM}} [ |B_t+a| ] - |a| = \mathbb E[ L_a^B(t)]$, for a Brownian motion $B$ of rate 2. By the variance-two assumption on the increments of $Y$, one can check that $\|g_0\|_{\ell^1} =1 .$ This whole discussion shows that whenever $\|g\|_{\ell^1}=1$, one has that $$\lim_{N\to \infty} \mathbb E [ N^{-1/2} \sum_{r=0}^{Nt} g(a_N+ Y(r)) ] = \mathbb E [ L_a^B(t)] =:h(t,a).$$
    Let $(U,V)$ be a joint limit point of $(N^{-1/2} Y_{Nt} , N^{-1/2} \sum_{r=0}^{Nt} f(a_N+ Y(r))). $ From the above limit, one can show that $V$ is a nondecreasing process, and $(U,V- L_a^U)$ must be a martingale in the joint filtration. This is because the above formula implies that $\mathbb E[ V(t) | \mathcal F_s] = h(t-s,U_s) = \mathbb E[ L_a^U(t)|\mathcal F_s].$ By the nondecreasing property of both $V$ and $L_a^U$, the difference $V-L_a^U$ is a martingale of bounded variation, thus we find that $V=L_a^U$, which is enough to imply the claim.
\end{proof}

\begin{lem}[Covariations of Brownian motions arising from strong mixing sequences] \label{mixl} Fix $m\in \mathbb N$, and let $(\Omega, A , \mathbb P)$ be some complete probability space. Let $(w_t^1, w^2_t,...,w^m_t)_{t\in \mathbb Z}$ be a (jointly) strictly stationary sequence on that probability space, with finite-range dependence. Assume that all coordinates have mean zero and finite moments of all orders. Then for any $T>0$ the $m$-tuple of processes given by $$N^{-1/2} \bigg(\sum_{r=0}^{Nt} w^1_r,..., \sum_{r=0}^{Nt} w^m_r \bigg)_{t\in [0,T]} $$ converges in law as $N\to \infty$, in the topology of $C([0,T],\Bbb R^m).$ The limit is given by an $m$-dimensional Brownian motion $(B^1,...,B^m)$ that is correlated, with covariance matrix at time 1 given by $(\Sigma_{ij})_{1\le i,j\le m}$ where $$\Sigma_{ij} = \sum_{t=-\infty}^\infty \mathrm{Cov} ( w^i_0,w^j_t ).$$
\end{lem}

\begin{proof} There are two parts to the proof: the tightness of each coordinate in $C[0,T]$, and the identification of the limit as a Brownian motion in $\mathbb R^m$ with the correct covariance.
\\
\\
\textbf{Step 1. Tightness.} We claim that there is a bound for each $j\in\{1,...,m\}$: $$\mathbb E \bigg[ \bigg( \sum_{r=s}^{t} w^j_r \bigg)^4 \bigg] \leq C(t-s)^2,$$ where $C$ is independent of $s,t \in \mathbb Z$ with $t>s$. This bound follows by multiplying out the fourth power, and using the $\rho$-mixing property. As an example, the mean-zero assumption implies that $|\mathbb E[ w_i w_j^3]| \leq C\ind_{\{|i-j|\le C\}}$ for some $C,\alpha>0$ independent of $i,j\in\mathbb Z$. This means that the bulk of the consequential terms will be those of the form $\mathbb E[\omega_i^2\omega_j^2]$, of which there are $O((t-s)^2)$ many. 

The above bound immediately gives $$\mathbb E \bigg[ \bigg( N^{-1/2} \sum_{r=Ns}^{Nt} w^j_r \bigg)^4 \bigg]^{1/4} \leq C|t-s|^{1/2},$$ which implies the tightness of each coordinate in $C[0,T]$ by Kolmogorov-Chentsov.
\\
\\
\textbf{Step 2. Identification of the limit point.} Simply from the assumptions and the result of the last step, it is clear that any limit point must be a L\'evy process in $\mathbb R^m$, with continuous paths. Thus it must be a Brownian motion with some covariance matrix and some drift. The drift is zero because of the mean-zero assumption on the weights. It remains to compute the covariance matrix $\Sigma.$ Since Step 1 proved the tightness using a uniform fourth moment bound, we have uniform integrability, and thus we can just compute the naive limit of the time-one covariances of the prelimiting processes. In other words, we have that $\Sigma_{ij} = \lim_{N\to\infty} \Sigma_{ij}^N$, where $$\Sigma_{ij}^N := \mathrm{Cov} \bigg( N^{-1/2}\sum_{r=0}^N w^i_r, N^{-1/2} \sum_{r=0}^N w^j_r\bigg) = N^{-1} \sum_{r,s=0}^N \mathrm{Cov} (w^i_r,w^j_s). $$ Use the joint stationarity of all coordinates of the sequence $\omega$, and we see that the right side can be rewritten as $N^{-1} \sum_{\ell=-N}^N (N-|\ell|) \mathrm{Cov} ( \omega_0^i , \omega_\ell^j) = \sum_{\ell=-N}^N (1-\frac{|\ell|}{N}) \mathrm{Cov} ( \omega_0^i , \omega_\ell^j).$ The assumptions imply that $|\mathrm{Cov} ( \omega_0^i , \omega_\ell^j)|\leq C\ind_{\{|j|\le C\}}$, from which the convergence to the expression in the statement of the lemma is clear.
\end{proof}

\begin{thm}[Limit theorem for all relevant processes]\label{b}
    Fix $k\in \mathbb N$ and $T>0$. Let $\vec y^N\in N^{-1/2}\mathbb Z^k$ such that $\vec y^N\to \vec y = (y_1,...,y_k)\in\mathbb R^k$ as $N\to\infty$. Let $\mathbf R=(R^1,...,R^k)$ be distributed as $\mathbf P_{\mathrm{RW}^{\otimes k}}.$ Under the measure $\mathbf P_{\mathrm{RW}^{\otimes k}},$ we have the joint convergence in distribution as $N\to \infty$: 
    \begin{align*}\bigg( N^{-1/2} & \mathbf R(Nt), \;N^{-1/2} \mathbf{Off}_{Nt}(\mathbf R, N^{1/2} \vec y^N), \;N^{-1/2} \mathbf{Diag}_{Nt}(\mathbf R) , \;\sum_{a=0}^5 |\mathbf{Err}^a_{Nt} (\mathbf R, N^{1/2}\vec y^N)| \bigg)_{t\in [0,T]} \\ \stackrel{(d)}{\longrightarrow} & \bigg( \mathbf W, \sigma^2 \sum_{1\le i<j\le k} L_0^{W^i-W^j+y_i-y_j} , v(W^1+...+W^k) + (\gamma^2-v^2)^{1/2} \sqrt k B, \mathbf 0\bigg).\end{align*} Convergence here is in the topology of $C([0,T],\mathbb R^{k+3}).$ Here $W^1,...,W^k, B$ are standard Brownian motions, independent of one another, and $\mathbf W:=(W^1,...,W^k)$. Furthermore $\mathbf 0$ is the zero process. Also, $\sigma$ is as defined in \eqref{sigma}, and $\mathbf{Off}_s$, $\mathbf{Diag}_s$, and $\mathbf{Err}_s$ were all defined in the discussion following \eqref{imp2} above. Furthermore $L_0^X$ denotes the local time at zero of the semimartingale $X$. Finally $v$ is the coefficient defined in \eqref{v}, and $\gamma$ was defined in \eqref{gamma}.
\end{thm}

\begin{proof} We break the proof into seven steps.
\\
\\
    \textbf{Step 1. Convergence of ``$\mathbf R$."} Marginally, it is a simple consequence of Donsker's invariance principle that $(N^{-1/2} \mathbf R(Nt))_{t\in [0,T]}$ converges in law (in the topology of $C[0,T])$ to a standard $k$-dimensional Brownian motion $\mathbf W.$
\\
\\
    \textbf{Step 2. Convergence of ``$\mathbf{Off}$."} We will show that for all pairs $(i,j)\in \{1,...,k\}^2$ with $i\ne j$, one has joint convergence of the triple as $N\to \infty$:
    \begin{align*}\bigg( N^{-1/2} & R^i(Nt), N^{-1/2} R^j(Nt) , N^{-1/2} \sum_{r_1,r_2=1}^{Nt}  \mathbf f\big( r_1-r_2, y_{i}^N-y_{j}^N + R^{i} (r_1) - R^{j}(r_2)\big)\bigg)_{t\in [0,T]} \\& \stackrel{(d)}{\longrightarrow} \big( X,Y , \sigma^2 \cdot L_0^{X-Y+y_i-y_j} \big),
    \end{align*}
 where $X,Y$ are two independent real-valued Brownian motions. This will be enough, as explained in Step 7 below. Without loss of generality assume that $(i,j)=(1,2)$. 

    Let us first write 
    \begin{equation}\label{decomp2}\sum_{r_1,r_2=1}^s \mathbf f(r_1-r_2, a + R^1(r_1)-R^2(r_2)) = \sum_{j=-\infty}^\infty O^a_j (s),
    \end{equation}
    where $$O_j^a(s):= \sum_{r=|j|}^{s-|j|} \mathbf f(j, a + R^1(r) - R^2(r+j))$$
    We will show for each $m\in \mathbb N$ the joint convergence of the $(m+2)$-tuple of processes $$(N^{-1/2} O^{a_N}_{-m} (Nt) ,..., N^{-1/2} O_m^{a_N}(Nt), N^{-1/2} R^1_{Nt} , N^{-1/2} R^2_{Nt} ) \stackrel{(d)}{\longrightarrow} (\lambda_{-m} L_0^{X-Y+a} ,..., \lambda_m L_0^{X-Y+a} , X,Y)$$
    whenever $N^{-1/2} a_N\to a$, where $X,Y$ are two independent Brownian motions and where $\lambda_j = \sum_{x\in \mathbb Z} \mathbf f(j,x).$ This is actually immediate from Lemma \ref{addfl}, since for any limit point, the law of the triple given by last two coordinates with any other coordinate has the right law, which characterizes the full $(m+2)$-tuple in this case, since any of the first $m$ coordinates is a deterministic path functional of the last two coordinates.
    
    Note that since 
    $\mathbf f$ is finitely supported, the above is enough to show the desired claim by taking $m$ large enough. 
\\
\\
    \textbf{Step 3. Convergence of ``$\mathbf{Diag}."$ } Here we prove joint convergence of the pair $$\big( N^{-1/2} \mathbf R(Nt), N^{-1/2}\mathbf{Diag}_{Nt}(\mathbf R))_{t\in [0,T]}\stackrel{(d)}{\longrightarrow} (\mathbf W, v(W^1+...+W^k)+(\gamma^2-v^2)^{1/2} \sqrt k B)$$ for some Brownian motion $B$ that is independent of the $k$-dimensional standard Brownian motion $\mathbf W$, and $v,\gamma$ are the real numbers from the theorem statement. This will be enough as explained in Step 7.

    Again we are going to use a ``slicing" argument as we did with $\mathbf{Off}$ in Step 2. More precisely, write 
    \begin{equation}\label{decomp1}\tfrac12 \sum_{r_1,r_2=0}^s \mathbf f( r_1-r_2, R^i(r_1) - R^i(r_2) ) \;\;- \;\;c_2^* s =: E_i(s)+\sum_{p=-\infty}^\infty D^i_p ( s) ,
    \end{equation}
    where 
    \begin{align*}D^i_p(s) &:= \tfrac12 \sum_{r=|p|}^{s-|p|}\big( \mathbf f( p, R^i(r) - R^i(r+p) )  - \mathbb E[ \mathbf f( p, R^i(r) - R^i(r+p))]\big) \\ E_i(s) &:= \tfrac12 \sum_{r_1,r_2=0}^s \mathbb E[\mathbf f( r_1-r_2, R^i(r_1) - R^i(r_2) )] \;\;- \;\;c_2^* s
    \end{align*}
    We will first show that $E_i$ is irrelevant in the limit. First note that it is nonrandom. Second, with $\eta_s$ as defined after Theorem \ref{mr}, note by conditioning that $\mathbb E[\mathbf f( p, R^i(r) - R^i(r+p))]$ does not depend on $r\in \mathbb Z$ and is in fact equal to $\mathbb E[ \eta_0\eta_p]$. Thus $E_i(s)=\sum_{r=-s}^s (s-|r|) \mathbb E[ \eta_0\eta_r]\;-\; c_2^* s.$ First note that $\sum_{r=-s}^s |r| \mathbb E[\eta_0\eta_r]$ converges absolutely as $s\to \infty$, by the strong mixing and mean zero property of $\eta$. In particular $\sum_{r=-s}^s |r| \mathbb E[\eta_0\eta_r]$ is $O(1)$ as $s\to \infty$, and can thus be disregarded due to the scaling of the process $\mathbf{Diag}$ by $N^{-1/2}$ as in the theorem statement. Thus we simply need to show that $s \big( \sum_{r=-s}^s \mathbb E[\eta_0\eta_r] \;- c_2^*)$ vanishes under the scalings. But by the expression for $c_2^*$ in Lemma \ref{3.1}, and the finite range dependence of the sequence $\eta_s$, this quantity is eventually zero for large enough $s$, as desired.
    
    Next, we show that for each $m\in \mathbb N$ the joint convergence of the $(2m+2)k$-tuple of processes $$\big( \big( N^{-1/2} D^i_{p} (Nt)\big)_{-m\le p \le m, 1\le i \le k} , N^{-1/2} \mathbf R(Nt) \big)_{t\in [0,T]} \stackrel{(d)}{\longrightarrow} \big( \big( B^i_p)_{-m\le p \le m, 1\le i\le k} , \mathbf W\big)$$
    where $\big( \big( B^i_p)_{-m\le p \le m, 1\le i\le k} , \mathbf W\big)$ is a Brownian motion in $\mathbb R^{(2m+2)k}$ whose covariations are given by 
    \begin{align*}
        \langle B^i_p ,B^j_q\rangle &=\tfrac14  t \delta_{ij} \sum_{a,b = -\infty}^\infty \mathrm{Cov} ( \mathbf f( p, R^i(a+p) -R^i(a)) , \mathbf f( q, R^i(b+q)-R^i(b))),
        \\ \langle B^i_p, W^j\rangle_t  &=\tfrac12 t\delta_{ij} \sum_{b = 0}^\infty \mathrm{Cov} ( R^i(1) -R^i(0)) , \mathbf f( p, R^i(b+p)-R^i(b))),
        \\ \langle W^i,W^j \rangle_t &=  t\delta_{ij}.
    \end{align*}
    The proof of this is actually immediate from the result of Lemma \ref{mixl} and a short calculation, noting in particular that the \textit{time increments} of the $(2m+2)k$-tuple of processes given by $\big( \big( N^{-1/2} D^i_{p} (Nt)\big)_{-m\le p \le m, 1\le i \le k} , N^{-1/2} \mathbf R(Nt) \big)$ do form a (jointly) strongly mixing sequence on the probability space of the random walk $(R^1,...,R^k)$ as needed to apply that lemma. In turn, this is because for a stationary sequence of random variables of the form $w^{ij}_t := \mathbf f(j, R^i(j+t)-R^i(t)) ,$ the random variables $w^{ij}_s,w^{ij}_t$ are independent unless $|t-s|\le |j|,$ simply by independence of the increments of the random walks $R^i$.

    With the above convergence proved for each $m$, one now takes a sum. 
    hen one matches the resulting covariation processes with the processes given in the theorem statement, noting that $v = \langle \sum_{p\in\mathbb Z} B^i_p , W^i\rangle_1$ for each $1\le i \le k$, and likewise $\gamma^2 = \langle \sum_{p\in\mathbb Z} B^i_p \rangle_1 $ for each $1\le i \le k$.
\\
\\
    \textbf{Step 4. Vanishing of the ``diagonal" error terms.} Here we will control $\mathbf{Err}^1$ and $\mathbf{Err}^2,$ showing that they converge to 0 in probability (in the topology of $C[0,T]).$ This is similar to Step 3, but now the extra factors of $N^{-1/4}$ yield convergence in probability to 0. 
    
    We explain in some detail the proof for $\mathbf{Err}^2$ since obtaining the constant drift of order $c_4^* - \frac12 \gamma^2$ is a nontrivial task. Define the function $g:\mathbb Z^8 \to \mathbb R$ by the formula $g(\vec r, \vec x):= \kappa_{(4)} ( \omega_{r_1,x_1} ,..., \omega_{r_4,x_4}).$ We are going to see that the appearance of the term $\frac12 \gamma^2$ in the drift part of $\mathbf{Err}^2$ is a consequence of the fact that the order-four cumulant does not commute with the random walk expectation. More precisely, we have $\mathbf E_{\mathrm{RW}} [ g(r_1,...,r_4, R(r_1), ..., R(r_4))] \neq \kappa_{(4)} ( \omega_{r_1, R(r_1)} ,..., \omega_{r_4, R(r_4)}) $ where the random walk $R$ is independent of the environment $\omega$. To calculate the difference of the two, one uses the so-called ``law of total cumulance" \cite{Bri}. Thanks to the mean-zero assumption on the weights $\omega$, we can let $\mathbf f(t,x):= \mathbb E[\omega_{0,0}\omega_{t,x}]$, and a simple calculation then shows that this difference is explicitly given by 
    \begin{align*}
        \mathbf E_{\mathrm{RW}} &[ g(r_1,...,r_4, R(r_1), ..., R(r_4))] - \kappa_{(4)} ( \omega_{r_1, R(r_1)} ,..., \omega_{r_4, R(r_4)}) \\&= \mathrm{Cov} \big(\mathbf f(r_1-r_2, R(r_1)-R(r_2)), \mathbf f(r_3-r_4, R(r_3)-R(r_4)\big) \\ &\; +\mathrm{Cov} \big(\mathbf f(r_1-r_3, R(r_1)-R(r_3)), \mathbf f(r_2-r_4, R(r_2)-R(r_4)\big)\\ &\;\;+\mathrm{Cov} \big(\mathbf f(r_1-r_4, R(r_1)-R(r_4)), \mathbf f(r_2-r_3, R(r_2)-R(r_3)\big).
    \end{align*}
    Sum this expression over $0\le r_1 ,r_2,r_3,r_4 \le Nt$, then subtract $\gamma^2t/2$, and then multiply the result by $\lambda^4=N^{-1}$. The resultant expression can be shown to converge to $0$ in the topology of $C[0,T]$, using similar arguments as in Step 3. The reason that the drift is $\gamma^2/2$, is that the right side of the above expression has three terms, and the fourth order part of the cumulant expansion already comes with a factor of $1/4!$ so in total we get $3/4! = 1/8.$ However, the definition \eqref{gamma} of $\gamma^2$ already comes with a factor of $1/4$, thus we only need an extra factor of $\frac12$ on $\gamma^2$ to obtain $\frac18$.

    Finally, to bound $\mathbf{Err}^2$, one simply uses the above argument $k$ separate times, for each of the $k$ independent random walk coordinates $R^1,...,R^k$, and we obtain a total drift of $k \times \frac12 \gamma^2 = \frac{k}2 \gamma^2$.

    With the proof for $\mathbf{Err}^2$ completed, let us now sketch the proof for the vanishing of $\mathbf{Err}^1$. One can ask why the drift term of $\mathbf{Err}^1$ is the ``naive" one just given by $c_3^*$ without any correction. The reason that one does not see the appearance of the extra term such as $\gamma^2$ is that (by the mean-zero assumption on $\omega$), the order-three cumulant actually \textit{does} commute with the random walk expectation, unlike the order-four case above. Thus there is no issue as was the case for $\mathbf{Err}^2$, and the proof can be done similarly as in Step 3 without any additional modifications needed.
\\
\\
    \textbf{Step 5. Vanishing of the ``off-diagonal" error terms.} Here we will control $\mathbf{Err}^3$ and $\mathbf{Err}^4$, showing that they converge to 0 in probability (in the topology of $C[0,T]).$ This is similar to Step 2, but now the extra factors of $N^{-1/4}$ yield convergence in probability to 0 rather than a nontrivial stochastic process. Noting that the higher order cumulants of $\omega$ are finitely supported in the sense that $|\kappa_{(n)} ( \omega_{t_1,x_1},..., \omega_{t_n,x_n}) | \leq C \ind_{\{ \sum_{1\le i<j\le n}|t_i-t_j| + |x_i-x_j|\le C\}},$ for $n=3,4$, the proof is extremely similar to Step 2 and is omitted. 
\\
\\
    \textbf{Step 6. Vanishing of the remainder error term.} Here we will control the last error term $\mathbf{Err}^5,$ showing that it converges to 0 in probability (in the topology of $C[0,T]).$ 

    From the assumption Assumption \ref{ass1} on the environment, and the fact that $\omega_{0,0}$ has exponential tails, we claim the bound $$\bigg| \sum_{j_1,...,j_n=1}^k \sum_{r_1,...,r_n=0}^s \kappa_{(n)} \big( \omega_{s-r_1 , y_{j_1}+\gamma^{j_1}(r_1)},...,\omega_{s-r_n , y_{j_n}+\gamma^{j_n}(r_n)}\big) \bigg| \leq C^n n! \cdot s,$$ where $C$ is independent of $n,s$ and paths $\boldsymbol \gamma:\mathbb Z\to \mathbb Z^k$. This bound would not be true with the absolute value inside the sum. To prove it, for any deterministically bounded sequence $(w_t)_t$ with finite range dependence, it is true that the $n^{th}$ cumulant of $\sum_1^s w_t$ grows linearly in $s$ with a prefactor of $C^n n!$, by \cite[Theorem 9.1.7]{Fer} (see Lemma \ref{appx} for a similar calculation). Apply this fact to the sequence $w_t: =\sum_{j=1}^k  \omega_{-t, y_j + \gamma^j (t)},$ and we get the above claim.
    
    Consequently we immediately obtain a bound \begin{equation}\label{err5}\sup_{\vec y} \sup_{\boldsymbol \gamma} | \mathbf{Err}^5_s ( \lambda; \boldsymbol\gamma; \vec y) |\leq s \cdot \sum_{n=5}^\infty (C\lambda)^n \leq C's \lambda^5
    \end{equation}
    where the last inequality is valid for sufficiently small $\lambda>0$. Thus, under the scalings $(\lambda, t) \to (N^{-1/4}, Nt)$, we actually obtain the deterministic bound 
    $$\sup_{\vec y} \sup_{t\in [0,T] \cap (N^{-1}\mathbb Z)} |\mathbf{Err}^5_{Nt} (N^{-1/4} ; \mathbf R ; \vec y) | \leq CN^{-1/4},
    $$
    which clearly implies the vanishing.
\\
\\
    \textbf{Step 7. Putting everything together.} Finally, we prove the theorem by putting together the results of Steps 1-6. The tightness of all coordinates has been proved in previous steps. Consider any joint limit point $(\mathbf W, L , D, E)$ of the 4-tuple of processes in the theorem statement. The result of step 2 implies that the marginal law of the pait $(\mathbf W,L)$ must be equal to $(\mathbf W,\sigma^2  \sum_{1\le i<j \le k} L_0^{W^i-W^j +a_i-a_j})$, which forces $L= \sigma^2  \sum_{1\le i<j \le k} L_0^{W^i-W^j +a_i-a_j}.$ The result of Step 3 implies that the marginal law of the pair $(\mathbf W, D)$ must be equal to $(\mathbf W, v(W^1+...+W^k ) +(\gamma^2-v^2)^{1/2} k^{1/2} B)$ for a standard $(k+1)$-dimensional Brownian motion $(\mathbf W, B).$ The result of Steps 4-6 shows that $E$ must be the zero process. Putting these facts together uniquely characterizes the full 4-tuple $(\mathbf W, L, D, E)$ and thus proves uniqueness of the limit point.
\end{proof}

With the limit theorem proved, we now prove some bounds on all the processes that will be needed for uniform integrability later.

\begin{prop}[Exponential moment bounds on all processes.]\label{c}
    There exists $C>0$ such that uniformly over all $t >0$ and $\lambda$ in a small neighborhood of the origin, we have a bound of the form $$\sup_{\vec y \in \mathbb Z^d} \mathbf E_{\mathrm{RW}^{\otimes k}} \bigg[ e^{\lambda \big(|\mathbf{Off}_t(\mathbf R,\vec y)| + |\mathbf{Diag}_t(\mathbf R)| + \sum_{a=0}^5 |\mathbf{Err}^a_t(\mathbf R,\vec y)| \big)} \bigg]\leq Ce^{C\lambda^2 t}. $$
\end{prop}
In particular, replacing $(\lambda, t)\to (N^{-1/2}\lambda, Nt)$ leaves the right side unchanged, which is important for uniform integrability under the relevant scaling of the processes as in Theorem \ref{b}.
\begin{proof} Again we have steps. 
\\
\\
\textbf{Step 1. Bounding ``$\mathbf{Diag}$."} First we state and prove the following useful bound. Assume $(w_t)_{t\in \mathbb Z}$ is a strictly stationary and fast mixing sequence of mean zero with $w_0$ deterministically bounded. More precisely assume $\|w_0\|_{L^\infty} \leq \alpha$ and assume a range of dependence at most $J\in \mathbb N$ for $(w_t)_t$. Then we have a bound 
\begin{equation}\label{useful}\mathbf E[ e^{\lambda (w_1+...+ w_t)} ] \leq 2e^{J^2\alpha \lambda^2 t},
\end{equation}
with $C$ a universal constant that is uniform over such sequences, and $t,J ,\alpha, \lambda >0$. To prove this bound, Taylor expand the exponential, and it amounts to showing that for positive integers $p$ we have a bound of the form $\mathbb E[ (w_1+...+w_t)^{2p}] \leq J^{p} \alpha^{2p} t^p,$ with $C$ independent of $t,p, J,\alpha$. Then one can just sum over $p$ to obtain the claim, using $e^{\lambda x} \leq 2 \sum_{p\ge 0}  \frac{\lambda^{2p}x^{2p} }{(2p)!}$, to verify the above bound. To prove the bound $\mathbb E[ (w_1+...+w_t)^{2p}] \leq J^{p}\alpha^{2p}  t^p,$ expand the $(2p)^{th}$ power and then use an argument similar to the proof of Step 1 in the proof of Lemma \ref{mixl}. The number of relevant terms will be on the order of the number of those $(2p)$-tuples of indices in $\{1,...,t\}$ that pair up with one another within distance $J$, which is indeed upper bounded by $ J^{p}  t^p. $ Each term in the sum can be upper bounded brutally by $\alpha^{2p}$ since $|w_i| \leq \alpha$, completing the proof of \eqref{useful}. 

Now, given the above bound, we are going to re-use some notations and facts from the proof of Theorem \ref{b}. Recall the decomposition \eqref{decomp1}, and recall that $E_i(s)$ was shown to be deterministic and $O(1)$, and is thus bounded independently of $s$ and can be disregarded. So we just need to show that $\sup_{a\in \mathbb Z} \mathbb E[ e^{\lambda \sum_{j=-\infty}^\infty D_j(t)}] \leq Ce^{C\lambda^2 t},$ where we have suppressed the superscript $i$ on $D^i_j$ to avoid confusion (one can just fix $i\in \{1,...,k\}$ as desired, henceforth). Choose a sequence $c(j)>0$ with $j\in \mathbb Z$ such that $\sum_j c(j) = 1$ and $|j|^2 c(j) \to 1$ as $j\to \infty$. Use Holder's inequality to say that $$ \mathbb E[ e^{\lambda \sum_{j=-\infty}^\infty D_j(t)}] \leq \prod_{j \in \mathbb Z} \mathbb E [ e^{\lambda c(j)^{-1} D_j(t)}]^{c(j)} \leq 2^{\sum_j c(j)} e^{\lambda^2 t\sum_j j^2 c(j)^{-1} \|\mathbf f(j,\bullet)\|_{\ell^\infty(\mathbb Z)}}.$$ In the last bound we used \eqref{useful} with $\lambda = c(j)$ and $J =j$, noting (as explained in Step 3 of the proof of Theorem \ref{b}) that the increments of $D_j$ do form a strong mixing sequence with a range of dependence equal to exactly $j$. Recall that $j\mapsto \|\mathbf f(j,\bullet)\|_{\ell^\infty(\mathbb Z)}$ is finitely supported by Assumption \ref{ass1}, thus the infinite sum in the exponential is equal to some finite constant, completing the proof.
\\
\\
\textbf{Step 2. Bounding ``$\mathbf{Off}$."} Recall \eqref{e2}, then sum over all indices $r_1\le ... \le r_m$, and we obtain that for $f:\mathbb Z \to \mathbb R_{\ge 0}$ of exponential decay, one has the bound 
\begin{align*}
    \sup_{a\in \mathbb R} \mathbf E_{\mathrm{RW} } \bigg[ \bigg( \sum_{r=0}^t f(R_r -a) \bigg)^m \bigg] &\leq m! \sup_{a\in \mathbb R}  \sum_{0\le r_1 \le ...\le r_m \le t} \mathbf E_{\mathrm{RW}}\bigg[  \prod_{j=1}^{m} f(R_{r_j} - a)\bigg] \\&\le m! \|f\|_{\ell^1(\mathbb Z)}^m \sum_{r_1 \leq ... \le r_m \le t} \prod_{j=0}^{m-1} 1\wedge (r_{j+1} - r_j)^{-1/2}.
\end{align*}
We claim that the sum over $r_1 \leq ... \le r_m \le t$ is upper-bounded by $C^m \|f\|_{\ell^1(\mathbb Z)}^m t^{m/2}/ (m/2)!$, where $(m/2)!$ should be interpreted using the Gamma function, and $C$ is independent of $t,m$. Indeed, if we multiply that sum by $t^{-m/2}$ it becomes a Riemann sum approximation for the iterated integral $\int_{0\le s_1 \le... \le 1} \prod_{j=0}^{m-1} (s_{j+1}-s_j)^{-1/2}ds_1 \cdots ds_m = 1/\Gamma(m/2).$ Consequently, \begin{align}
    \notag \sup_{a\in \mathbb R} \mathbf E_{\mathrm{RW} } & \bigg[ e^{\lambda \sum_{r=0}^t f(R_r -a)} \bigg] \\&\le \notag \sup_{a\in \mathbb R}  \sum_{m=0}^\infty \frac{\lambda^m}{m!} 
    \mathbf E_{\mathrm{RW} } \bigg[ \bigg( \sum_{r=0}^t f(R_r -a) \bigg)^m \bigg]\\& \notag \leq \sum_{m=0}^\infty \frac{C^m t^{m/2}\|f\|_{\ell^1(\mathbb Z)}^m}{(m/2)!} \\ &\leq 2 e^{C^2 \lambda^2 \|f\|_{\ell^1(\mathbb Z)}^2 t}.\label{b2}
\end{align}
Now, given the above bound, we are going to re-use some notations and facts from the proof of Theorem \ref{b}. Recall the decomposition \eqref{decomp2}. So we just need to show that $\sup_{a\in \mathbb Z} \mathbb E[ e^{\lambda \sum_{j=-\infty}^\infty O_j^a(t)}] \leq Ce^{C\lambda^2 t}.$ Choose a sequence $c(j)>0$ with $j\in \mathbb Z$ such that $\sum_j c(j) = 1$ and $|j|^2 c(j) \to 1$ as $j\to \infty$. Use Holder's inequality to say that $$ \mathbb E[ e^{\lambda \sum_{j=-\infty}^\infty O_j^a(t)}] \leq \prod_{j \in \mathbb Z} \mathbb E [ e^{\lambda c(j)^{-1} O^a_j(t)}]^{c(j)} \leq 2^{\sum_j c(j)} e^{\lambda^2 t\sum_j j^2 c(j)^{-1} \|\mathbf f(j,\bullet)\|_{\ell^1(\mathbb Z)}^2}.$$ In the last bound we used \eqref{b2} with $\lambda = c(j)$ and $f = \mathbf f(j,\bullet)$. Recall that $j\mapsto \|\mathbf f(j,\bullet)\|_{\ell^1(\mathbb Z)}^2$ is finitely supported by Assumption \ref{ass1}, thus the infinite sum in the exponential is equal to some finite constant, completing the proof.
\\
\\
\textbf{Step 3. Bounding the errors. } The proof to bound the ``diagonal" error terms $\mathbf{Err}^1$ and $\mathbf{Err}^2$ is very similar to the proof of bounding $\mathbf{Diag}$ in Step 1. The proof to bound the ``off-diagonal" error terms $\mathbf{Err}^3$ and $\mathbf{Err}^4$ is very similar to the proof to bound $\mathbf{Off}$ in Step 2. The last error term $\mathbf{Err}^5_t$ is deterministically bounded above by $C\lambda^5 t$ as proved in \eqref{err5}, stronger than required.
\end{proof}

\section{Proof of the main result}

Letting $Z^N_{s,t}(x,y)$ be as in \eqref{zn}, we now write out a relation for $Z^N$, valid for integer times $s<t<u$ and for $a,b\in \mathbb Z$:
$$Z^N_{s,u} ( a,b) = \sum_{y\in\mathbb Z} Z^N_{s,t}(a,y)Z^N_{t,u} ( y,b).$$
For the rescaled version $\mathcal Z^N_{s,t}$ in Theorem \ref{mr}, this relation yields the \textbf{propagator equation}, valid for times $s<t<u$ with $s,t,u\in N^{-1} \mathbb Z$ and for $a,b\in N^{-1/2} \mathbb Z$ 
\begin{equation}\label{propa}\mathcal Z^N_{s,u} ( a,b) = N^{-1/2} \sum_{y\in N^{-1/2} \mathbb Z} \mathcal Z^N_{s,t}(a,y)\mathcal Z^N_{t,u} ( y,b).
\end{equation}

\begin{prop}[Limit of the moments]\label{limmom} Fix $k\in \mathbb N$ and $\phi:\mathbb R\to\mathbb R$ continuous and bounded. Let $\vec y^N\in N^{-1/2}\mathbb Z^k$ such that $\vec y^N\to \vec y = (y_1,...,y_k)$ as $N\to\infty$, and let $t_N\in N^{-1}\mathbb Z$ such that $t_N\to t>0$ as $N\to \infty$. We have that 
\begin{align*}\lim_{N\to\infty} N^{-k/2} \sum_{x_1,...,x_k\in N^{-1/2}\mathbb Z}\mathbb E[ \mathcal Z^N_{0,t_N}&(x_1,y_1^N)\cdots \mathcal Z^N_{0,t_N}(x_k,y_k^N)] \prod_{j=1}^k \phi(x_j) \\&= \mathbf E_{\mathrm{BM}^{\otimes k}}^{(y_1,...,y_k)} \bigg[ e^{\sigma^2 \sum_{i<j} L_0^{W^i-W^j} (t) } \prod_{j=1}^k \phi(W^j_t+vt) \bigg].
\end{align*}
    Here $\sigma$ is as in \eqref{sigma}, and $v$ is as in \eqref{v}. The expectation on the right side is with respect to a $k$-dimensional Brownian motion $(W^1,...,W^k)$ started from $(y_1,...,y_k)$.
\end{prop}

\begin{proof}
The left side can be rewritten in terms of the processes $\mathbf{Diag}, \mathbf{Off}, \mathbf{Err}^a$ using Proposition \ref{a}. To take the limit of that expression, one uses the invariance principle of Theorem \ref{b} (taking this limit is allowed thanks to the uniform integrability guaranteed by the bound in Proposition \ref{c}). Notice that a naive limit will yield  $$\mathbf E_{\mathrm{BM}^{\otimes (k+1)}}^{(y_1,...,y_k, 0)} \bigg[ e^{\sigma^2 \sum_{i<j} L_0^{W^i-W^j} (t)  + v( W_t^1+... +W_t^k) +(\gamma^2-v^2)^{1/2}k^{1/2} B_t- \frac{k\gamma ^2t}2} \prod_{j=1}^k \phi(W^j_t) \bigg],$$ where the expectation is over a standard $(k+1)$-dimensional Brownian motion $(W^1,...,W^k,B)$ started from $(y_1,...,y_k,0).$ Using the independence of $B$ from $(W^1,...,W^k)$ we can condition out $B$, meaning that $$\mathbb E[ e^{(\gamma^2-v^2)^{1/2}k^{1/2} B_t - \frac12 (\gamma^2-v^2) kt} | W^1,...,W^k]=1.$$ Thus the above expression can be simplified to 
$$\mathbf E_{\mathrm{BM}^{\otimes k}}^{(y_1,...,y_k)} \bigg[ e^{\sigma^2 \sum_{i<j} L_0^{W^i-W^j} (t)  + v( W_t^1+... +W_t^k) - \frac{kv ^2t}2} \prod_{j=1}^k \phi(W^j_t) \bigg], $$where the independent Brownian motion $B$ has been conditioned out. Now an application of Girsanov's theorem gives the desired form, as in the statement of the proposition.
\end{proof}

We also have the following version of Proposition \ref{limmom}, where we allow for three times $(0,t,u)$ instead of just two times $(0,t)$.

\begin{prop}[Limit of the moments, generalized]\label{limmom2}
    Fix $k\in \mathbb N$ and let $\phi, \psi:\mathbb R\to\mathbb R$ continuous and bounded. Let $\vec y^N\in N^{-1/2}\mathbb Z^k$ such that $\vec y^N\to \vec y = (y_1,...,y_k)$ as $N\to\infty$, and let $t_N,u_N\in N^{-1}\mathbb Z$ such that $t_N\to t>0$ and $u_N\to u>t$ as $N\to \infty$. We have that 
\begin{align*}\lim_{N\to\infty} &N^{-k} \sum_{\substack{w_1,...,w_k \in N^{-1/2}\mathbb Z\\x_1,...,x_k\in N^{-1/2}\mathbb Z}}\mathbb E\bigg[ \prod_{j=1}^k \mathcal Z^N_{0,t_N}(w_j,x_j)\prod_{j=1}^k \mathcal Z^N_{t_N,u_N}(x_j,y_j^N)\bigg] \prod_{j=1}^k \phi(w_j)\psi(x_j) \\&= \mathbf E_{\mathrm{BM}^{\otimes k}}^{(y_1,...,y_k)} \bigg[ e^{\sigma^2 \sum_{i<j} L_0^{W^i-W^j} (u) } \prod_{j=1}^k \psi(W^j_t+vt)
\prod_{j=1}^k \phi (
W^j_u+vu) \bigg].
\end{align*}
\end{prop}

Notice that by taking $\psi=1$ and using the propagator equation \eqref{propa}, the result of Proposition \ref{limmom2} actually implies the result of Proposition \ref{limmom}, thus it is strictly more general. We could formulate an even more general version with $m$ times rather than two or three, but we will not need it, and the proof for three times already gives the general idea.

\begin{proof}
    The proof follows the same overall strategy as the proof of Proposition \ref{limmom}. We will sketch the strategy. Similar to what we did in deriving \eqref{IMP}, one can write out the definitions to yield 
    \begin{align*} &\sum_{\substack{w_1,...,w_k \in \mathbb Z\\x_1,...,x_k\in \mathbb Z}}\mathbb E\bigg[ \prod_{j=1}^k  Z^N_{0,t}(w_j,x_j)\prod_{j=1}^k  Z^N_{t,u}(x_j,y_j^N)\bigg] \prod_{j=1}^k \phi(w_j)\psi(x_j)
        \\&=\mathbf E_{\mathrm{RW}^{\otimes k}} \bigg[ e^{\log\mathbb E \big[ e^{\lambda \sum_{j=1}^k  \sum_{r=0}^{u} \omega_{u-r, y_j+R^j(r)}}|R^1,...,R^k\big]} \cdot \prod_{j=1}^k \psi(y_j^N+R^j(t)) \phi(y_j^N+R^j(u))\bigg] .
    \end{align*}
    Then, exactly as before, Taylor expand $$\lambda \mapsto  \log\mathbb E \big[ e^{\lambda \sum_{j=1}^k  \sum_{r=0}^{u} \omega_{u-r, y_j+R^j(r)}}|R^1,...,R^k\big]$$ and a very similar case-by-case analysis for each term of the expansion (as done in Section 3) will lead to a representation of the form given in Proposition \ref{a}. Then one proves the limit theorems as in Section 4, which will lead to the required limit as done in Proposition \ref{limmom}. The details are very similar and omitted.
\end{proof}

We are now ready to start putting together the proof of Theorem \ref{mr}. First we set some notations. 

\begin{defn}We denote by $\mathcal M(\mathbb R^2)$ the space of locally finite and nonnegative Borel measures on $\mathbb R^2$. Equip $\mathcal M(\mathbb R^2)$ with the weakest topology under which the maps $\mu \mapsto \int_{\mathbb R^2} f\; d\mu$ are continuous for all $f\in C_c^\infty(\mathbb R^2)$.\end{defn}

Such a space turns out to be Polish by explicit construction of a complete separable metric. This choice of topology of $\mathcal M(\mathbb R^2)$ is so weak that one may verify that a family $M_N$ of $\mathcal M(\mathbb R^2)$-valued random variables is tight if and only if $\int_{\mathbb R^2} f\; dM_N$ is a tight family of $\mathbb R$-valued random variables, for all $f\in C_c^\infty(\mathbb R^2)$. The proof of this is left to the reader.

We also define $\mathbb R^4_{\uparrow}:= \{ (s,t,x,y) \in \mathbb R^4: s<t\}$, and $C(\mathbb R^4_{\uparrow})$ to be the set of continuous functions from $\mathbb R^4_\uparrow\to\mathbb R$. Equip $C(\mathbb R^4_{\uparrow})$ with the topology of uniform convergence on compact sets. Below, $\mathbb Q$ will denote rational numbers, but everywhere it could be replaced by any countable dense subset of $\mathbb R$ as desired.

\begin{proof}[Proof of Theorem \ref{mr}.]
    We are going to use the following classification theorem for the equation \eqref{she}, which was recently proved in \cite{Par24a}, heavily inspired by flow-based ideas developed for the 2d SHE by \cite{Ts24}:

    Let $M_{s,t}$ be a family of $\mathcal M(\mathbb R^2)$-valued random variables indexed by $-\infty <s\le t<\infty$ with $s,t\in\mathbb Q,$ all defined on some complete probability space $(\Omega,\mathcal F,\mathbb P)$. Suppose that the following are true: 
    \begin{enumerate}

    \item $M_{s_j,t_j}$ are independent under $\mathbb P$ whenever $(s_j,t_j)$ are disjoint intervals.

        \item For all $f\in C_c^\infty(\mathbb R^2)$ and all $\psi\in C_c^\infty(\mathbb R)$ with $\int_\mathbb R\psi=1$, we have that $$\int_{\mathbb R^4 } \epsilon^{-1} \psi(\epsilon^{-1}(y_1-y_2))\cdot f(x,z) M_{t,u}(dy_1,dz) M_{s,t}(dx,dy_2)  \to \int_{\mathbb R^2} f(x,z) M_{s,u}(dx,dz)$$ in probability as $\epsilon\to 0$, whenever $s<t<u$.

        \item For $k\in \mathbb N$ and $\phi,\psi\in C_c^\infty(\mathbb R)$ we have that \begin{align*}\mathbb E\bigg[ \bigg( \int_{\mathbb R^2}& \phi(x) \psi(y)M_{s,t}(dx,dy) \bigg)^k \bigg] \\&= \int_{\mathbb R^{n}} \prod_{j=1}^k \phi(x_j) \mathbf E_{\mathrm{BM}^{\otimes k}}^{(x_1,...,x_k)} \bigg[\prod_{j=1}^k \psi(W^j_{t-s}+v(t-s)) e^{\sigma \sum_{1\le i<j\le k} L_0^{W^i-W^j}(t-s)} \bigg] dx_1 \cdots dx_k.
        \end{align*}
        The expectation on the right side is with respect to a $k$-dimensional Brownian motion $(W^1,...,W^k)$ started from $(x_1,...,x_k)$, and $L_0^X$ denotes the local time at 0 of $X$.
    \end{enumerate}
    Then there exists a $C(\mathbb R^4_\uparrow)$-valued random variable $\tilde M$ defined on the same probability space, such that $(M_{s,t},f) = \int_{\mathbb R^2} f(x,z) \tilde M_{s,t}(x,z)dxdz$ almost surely, for all $(s,t)\in\mathbb Q^2$ (with $s<t$) and $f\in C_c^\infty(\mathbb R^2)$. 
    Furthermore, there exists a space-time white noise $\xi$ on $\mathbb R^2$, defined on the same probability space, with the property that $\tilde M_{s,t}$ are the \textit{propagators} of \eqref{she} driven by $\xi$: for all $s,x\in \mathbb R$ the function $(t,y)\mapsto \tilde M_{s,t}(x,y)$ is a.s. the It\^o-Walsh solution of the equation \eqref{she} started from time $s$ with initial condition $\delta_x(y)$. In particular, the finite-dimensional marginal laws $(M_{s_1,t_1},...,M_{s_m,t_m})$ are $\mathrm{uniquely\; determined}$ by Conditions (1) - (3), for any finite collection of indices $s_j\le t_j$.

    We divide the remainder of the proof into steps, first showing tightness and then verifying the three items individually.
\\
\\
    \textbf{Step 1. Tightness in $\mathcal M(\mathbb R^2)^{\{ (s,t)\in \mathbb Q^2: s\le t\}}.$} View each $\mathcal Z^N_{s,t}$ as an element of $\mathcal M(\mathbb R^2)$ by viewing it as a superposition of Dirac masses of weight $N^{-1} \mathcal Z^N_{s,t}(x,y)$ at each vertex $(x,y) \in (N^{-1/2} \mathbb Z)^2$. For each fixed and bounded sequence of $(s^N,t^N) \in (N^{-1}\mathbb Z_{\ge 0})^2$ the family of measures $(\mathcal Z^N_{s^N,t^N})_{N\ge 1}$ is tight in $\mathcal M(\mathbb R^2)$ because Proposition \ref{limmom} implies that $\mathbb E[ \mathcal Z^N_{s^N,t^N}(f)^k]$ remains bounded as $N\to \infty$, which in particular implies that $\{\mathcal Z_{s,t}^N(f)\}_N$ is a tight family of $\mathbb R$-valued random variables for each $f\in C_c^\infty(\mathbb R^2)$. 

    Thus for any finite sequence of indices $(s_1,t_1),..., (s_m,t_m)$ the $m$-tuple $(\mathcal Z^N_{s_1^N,t_1^N} ,..., \mathcal Z^N_{s^N_m,t^N_m})$ is tight in $\mathcal M(\mathbb R^2)^m$. By Kolmogorov's extension theorem and a diagonal argument, one can now take the projective limit of any consistent family of finite-dimensional marginal laws which are themselves limit points of the finite-dimensional marginals $(\mathcal Z^N_{s_1,t_1},...,\mathcal Z^N_{s_m,t_m})$ as $N\to\infty$. This yields the existence of at least one limit point, which is a probability measure on the product $\sigma$-algebra of the underlying probability space $\mathcal M(\mathbb R^2)^{\{ (s,t)\in \mathbb Q^2: s\le t\}}.$

    Next we need to argue the uniqueness of the limit point, which will be done in Steps 2-4 below by showing that the limit point must satisfy the three listed items above. 
\\
\\
    \textbf{Step 2. Verification of (1).} Recall the strong mixing condition Assumption \ref{ass1}. Let $(s_j,t_j)$ be disjoint intervals for $1\le j \le m$. Assume they are ordered so that $t_j< s_{j+1}.$ These intervals have some positive gap $c>0$ between one another, i.e., $t_j+c\le s_{j+1}$ for $1\le j \le m$. Letting $\mathscr G_N(s,t)$ denote the $\sigma$-algebra generated by $\{ Z^N_{a,b}(x,y): Ns\le a<b\le Nt;\; x,y\in \mathbb Z\}$, we see from Assumption \ref{ass1} that $\mathscr G_N(-\infty, t_j)$ and $\mathscr G_N(s_{j+1},\infty))$ are independent for sufficiently large $N$. This easily implies that for any limit point $M$, the random measures $M_{s_j,t_j}$ are independent of one another.
\\
\\
    \textbf{Step 3. Verification of (2).} This is the more difficult item to verify. Recall the propagator equation \eqref{propa}. We will show that if $s_N <t_N<u_N$ with $(s_N,t_N,u_N)\to (s,t,u)$, then \begin{align*}\limsup_{\e\to 0} \limsup_{N\to\infty} \mathbb E \bigg[ \bigg(\int_{\mathbb R^4 } f(x,z) \epsilon^{-1} \psi(\epsilon^{-1}(y-w))&\mathcal Z^N_{t_N,u_N}(dy,dz) \mathcal Z^N_{s_N,t_N}(dx,dw) \\&- \int_{\mathbb R^2} f(x,z) \mathcal Z^N_{s_N,u_N}(dx,dz)\bigg)^2 \bigg]=0.
    \end{align*}
    Henceforth we suppress the subscripts on $s_N,t_N,u_N$ but it should be understood that they depend on $N$ in all prelimiting expressions.
    Multiply out the expectation and we obtain 
    \begin{align*}
        \bigg(\int_{\mathbb R^4 } &f(x,z) \epsilon^{-1} \psi(\epsilon^{-1}(y-w))\mathcal Z^N_{t,u}(dy_1,dz) \mathcal Z^N_{s,t}(dx,dy_2) - \int_{\mathbb R^2} f(x,z) \mathcal Z^N_{s,u}(dx,dz)\bigg)^2 \\ &= \bigg( \int_{\mathbb R^4} f(x,z) \bigg[ \epsilon^{-1} \psi(\epsilon^{-1} (y-w)) - N^{1/2} \ind_{\{y=w\}} \bigg]\mathcal Z^N_{t,u}(dy,dz) \mathcal Z^N_{s,t}(dx,dw)\bigg)^2 \\&= \int_{\mathbb R^8} \prod_{j=1,2} f(x_j,z_j) \bigg[ \epsilon^{-1} \psi(\epsilon^{-1} (y_j-w_j)) - N^{1/2} \ind_{\{y_j=w_j\}}\bigg] \prod_{j=1,2} \mathcal Z^N_{t,u}(dy_j,dz_j) \mathcal Z^N_{s,t}(dx_j,dw_j).
    \end{align*} We want to take expectation of this and let $N\to \infty$. Expand the product $\prod_{j=1,2} \big[ \epsilon^{-1} \psi(\epsilon^{-1} (y_j-w_j)) - N^{1/2} \ind_{\{y_j=w_j\}}\big]$ into four separate terms. From Proposition \ref{limmom2}, taking the limit of the expectation of each of those four terms is actually immediate, and in total, the limit equals 
    \begin{align*}
        \int_{\mathbb R^8} &\prod_{j=1,2} f(x_j,z_j) \bigg[ \epsilon^{-1} \psi(\epsilon^{-1} (y_j-w_j)) - \delta_0(y_j-w_j)\bigg] \mathbf E\bigg[\prod_{j=1,2} U_{t,u}(y_j,z_j) U_{s,t}(x_j,w_j)\bigg] \prod_{j=1,2} dy_jdz_jdx_jdw_j \\&= \mathbf E \bigg[ \bigg(\int_{\mathbb R^4 } f(x,z) \epsilon^{-1} \psi(\epsilon^{-1}(y-w)) U_{t,u}(y,z) U_{s,t}(x,w)dydzdxdw - \int_{\mathbb R^2} f(x,z) U_{s,u}(x,z)dxdz\bigg)^2 \bigg],
    \end{align*}
    where $U_{s,t}$ are the propagators for \eqref{she}, not necessarily defined on the same probability space as the environment $\omega$. Then we can just take $\epsilon\to 0$ in this quantity, and the result would be zero, since this is a known statement for the continuum object $U$ given by \eqref{she}.
\\
\\
    \textbf{Step 4. Verification of (3).} Given any limit point $(M_{s,t})_{(s,t)\in \mathbb Q_{\le}^2},$ the limiting moment formula of Proposition \ref{limmom} and the uniform integrability guaranteed by Proposition \ref{c} immediately gives the verification of Item (3). That proposition only considers $(s,t)=(0,t)$, but the general case is immediate from the time-stationarity of the field $\omega_{t,x}$, which implies $\mathcal Z^N_{a,b} \stackrel{(d)}{=} \mathcal Z^N_{a+h,b+h}$ for all (fixed) values of $a,b,h.$ This completes the proof of the theorem.
\end{proof}

Finally, we prove \eqref{endp}

\begin{proof}[Proof of \eqref{endp}]
    The point is to use Theorem \ref{mr} in conjunction with Proposition \ref{limmom}. Suppose $y^N\to y$ with $y^N\in N^{-1/2}\Bbb Z$ and $s_N,t_N\in N^{-1}\Bbb Z$ with $s_N\le t_N$. For bounded functions $\phi,\psi$ on $\Bbb R$, define random variables $A_N^{\phi, \psi} := N^{-1} \sum_{(a,b)\in \Bbb Z^2} \mathcal \phi(a)\mathcal Z^N_{s_N,t_N}(a,b)\psi(b).$ Let $\varrho_\delta^y(b):=\delta^{-1}\varrho(\delta^{-1}(b-y)),$ for some $\varrho\in C_c^\infty(\Bbb R)$ with $\int \varrho=1$. By Proposition \ref{limmom},  $$\lim_{\delta\to 0} \limsup_{N\to\infty} \Bbb E\big[ \big( A_N^{\phi, \varrho_\delta^y} - A_N^{\phi,N^{1/2}\ind_{\{y^N\}}}\big)^2\big]=0,$$ for any fixed $\phi \in C_c^0(\Bbb R)$. This is because the inner limit can be computed explicitly using Proposition \ref{limmom}, then it can be written in terms of the propagators $U_{s,t}$ for \eqref{she}, similarly to the proof Theorem \ref{mr} just above. Now the claim follows easily using the result of Theorem \ref{mr}, and the fact that $\int_{\Bbb R^2} \phi(a) U_{s,t} (a,b) \varrho^y_\delta (b) dadb \stackrel{L^2}{\to} \int_\Bbb R \phi(a) U_{s,t}(a,y) da$ as $\delta\to 0$.
\end{proof}

\bibliographystyle{alpha}
\bibliography{ref.bib}
\end{document}